\documentclass[reqno,11pt]{amsart}
\usepackage{amsthm, amsmath, amssymb}
\usepackage[utf8]{inputenc}
\usepackage[T1]{fontenc}
\usepackage[hidelinks]{hyperref}
\usepackage{microtype}
\usepackage{bm}
\usepackage[margin=1in]{geometry}
\usepackage{tikz}
\usepackage{enumerate}

\allowdisplaybreaks

\usepackage[noabbrev,capitalize,sort]{cleveref}
\crefname{equation}{}{}

\newtheorem{theorem}{Theorem}
\newtheorem{proposition}[theorem]{Proposition}
\newtheorem{lemma}[theorem]{Lemma}
\newtheorem{claim}[theorem]{Claim}
\newtheorem{corollary}[theorem]{Corollary}
\newtheorem{conjecture}[theorem]{Conjecture}

\theoremstyle{definition}
\newtheorem{definition}[theorem]{Definition}
\newtheorem{problem}[theorem]{Problem}
\newtheorem{question}[theorem]{Question}
\newtheorem{example}[theorem]{Example}

\theoremstyle{remark}
\newtheorem*{remark}{Remark}

\numberwithin{theorem}{section}

\newcommand{\abs}[1]{\left\lvert#1\right\rvert}

\newcommand{\floor}[1]{\left\lfloor #1 \right\rfloor}
\newcommand{\ceil}[1]{\left\lceil #1 \right\rceil}
\newcommand{\paren}[1]{\left( #1 \right)}

\newcommand{\set}[1]{\left\{ #1 \right\}}

\renewcommand{\epsilon}{\varepsilon}

\DeclareMathOperator{\sgn}{sgn}

\DeclareMathOperator{\prt}{part}

\newcommand{\PP}{\mathbb{P}}
\newcommand{\CC}{\mathbb{C}}

\newcommand{\R}{\mathbb{R}}

\newcommand{\cC}{\mathcal C}

\newcommand{\cE}{\mathcal E}
\newcommand{\cF}{\mathcal F}

\newcommand{\cP}{\mathcal P}

\newcommand{\cU}{\mathcal U}
\newcommand{\cV}{\mathcal V}

\newcommand{\cdim}[2][x]{\dim_{#1}^{\mathbb{C}}(#2)}

\title[Multilevel polynomial partitioning and semialgebraic hypergraphs]{Multilevel polynomial partitioning and semialgebraic hypergraphs: regularity, Tur\'an, and Zarankiewicz results}

\author[Tidor]{Jonathan Tidor}
\author[Yu]{Hung-Hsun Hans Yu}

\address{Department of Mathematics, Stanford University, Stanford, CA 94305, USA}
\email{jtidor@stanford.edu}

\address{Department of Mathematics, Princeton University, Princeton, NJ 08544, USA}
\email{hansonyu@princeton.edu}

\thanks{Tidor was supported by a Stanford Science Fellowship}

\begin{document}

\begin{abstract}
We prove three main results about semialgebraic hypergraphs. First, we prove an optimal and oblivious regularity lemma. Fox, Pach, and Suk proved that the class of $k$-uniform semialgebraic hypergraphs satisfies a very strong regularity lemma where the vertex set can be partitioned into $\mathrm{poly}(1/\varepsilon)$ parts so that all but an $\varepsilon$-fraction of $k$-tuples of parts are homogeneous (either complete or empty). Our result improves the number of parts in the partition to $O_{d,k}((D/\varepsilon)^{d})$ where $d$ is the dimension of the ambient space and $D$ is a measure of the complexity of the hypergraph; additionally, the partition is oblivious to the edge set of the hypergraph. We give examples that show that the dependence on both $\varepsilon$ and $D$ is optimal.

From this regularity lemma we deduce the best-known Tur\'an-type result for semialgebraic hypergraphs.
Third, we prove a Zarankiewicz-type result for semialgebraic hypergraphs. Previously Fox, Pach, Sheffer, Suk, and Zahl showed that a $K_{u,u}$-free semialgebraic graph on $N$ vertices has at most $O_{d,D,u}(N^{2d/(d+1)+o(1)})$ edges and Do extended this result to $K_{u,\ldots,u}^{(k)}$-free semialgebraic hypergraphs. We improve upon both of these results by removing the $o(1)$ in the exponent and making the dependence on $D$ and $u$ explicit and polynomial.

All three of these results follow from a novel ``multilevel polynomial partitioning scheme'' that efficiently partitions a point set $P\subset\mathbb{R}^d$ via low-complexity semialgebraic pieces. We prove this result using the polynomial method over varieties as developed by Walsh which extends the real polynomial partitioning technique of Guth and Katz.

We give additional applications to the unit distance problem, the Erd\H{o}s--Hajnal problem for semialgebraic graphs, and property testing of semialgebraic hypergraphs.
\end{abstract}

\maketitle

\section{Introduction}

Semialgebraic graphs and hypergraphs appear ubiquitously in discrete geometry. These are graphs and hypergraphs whose vertex set is a set of points in $\R^d$ and whose edges are defined by a Boolean combination of polynomial inequalities and equalities. For example, a foundational result in incidence geometry is the Szemer\'edi--Trotter theorem \cite{ST83} on the maximum number of incidences between a set of points and lines in $\R^2$. Representing a line by the pair $(m,b)$ of slope and $y$-intercept, we know that the incidence of a point $(x,y)$ with this line is given by the polynomial equality $y=mx+b$. Thus we see that the point-line incidence problem can be represented by a semialgebraic graph. As another example, one of the most notorious open problems in discrete geometry is the Erd\H{o}s unit distance problem \cite{Erd46}. This problem asks, given $n$ points in $\R^2$, what is the maximum number of pairs of points at distance 1? Since the unit distance condition can be written as $(x-x')^2+(y-y')^2=1$, we see that the unit distance problem can also be represented by a semialgebraic graph.

Two parameters are important to measure the complexity of a semialgebraic representation. The first is the dimension $d$ in which the vertex set lies. The second is the total degree, defined to be the sum of the degrees of the polynomials used in the definition of the edge set. In both of the previous examples, the dimension and total degree are both 2. As every hypergraph has a semialgebraic representation either with very large dimension or very large total degree, it will be important to quantify our results in terms of these parameters. Previous works in this area typically studied semialgebraic graphs and hypergraphs of bounded dimension and total degree. In contrast, our methods are able to handle those of slowly growing total degree.

In this paper, we prove three main results about semialgebraic hypergraphs. The first concerns the regularity lemma.

\subsection{Regularity lemmas}
A central tool in extremal graph theory is Szemer\'edi's graph regularity lemma \cite{Sze78}. This result gives a structural decomposition of an arbitrary graph, partitioning the vertex set into a bounded number of parts so that between almost all pairs of parts the graph behaves quasirandomly. Crucially, the number of parts in the partition does not depend on the number of vertices of the graph, but only on the error parameter $\epsilon$. Szemer\'edi's regularity lemma is extremely versatile, finding numerous applications in extremal graph theory as well as in additive combinatorics and other fields (see, e.g., the surveys \cite{CF13,KS96}). One drawback of the regularity lemma is that the number of parts, though independent of the number of vertices, is quite large. Szemer\'edi's original proof gives a tower of twos of height $\epsilon^{-O(1)}$ and a construction of Gowers shows that this tower-type behavior is necessary \cite{Gow97}.

When studying hypergraphs the situation becomes even more complicated. Different hypergraph generalizations of Szemer\'edi's regularity lemma have been proved by Gowers, R\"odl et al., and Tao \cite{Gow07,RNSSK05,Tao06}. However, in order for these results to be strong enough to prove useful applications (in particular the hypergraph removal lemma) the statement of the result and the type of structure produced is necessarily quite complicated. The quantitative dependence also is worse; for $k$-uniform hypergraphs, the number of parts is $k-1$ steps up the Ackermann hierarchy and this enormous growth rate is necessary \cite{MS19}.

In the last twenty years, a fruitful line of research has been to develop simpler and more efficient regularity lemmas for restricted families of graphs and hypergraphs. Since in discrete geometry applications one typically applies the regularity lemma to a geometrically defined graph, these restricted regularity lemmas have numerous applications in this field. One of these extensions is to graphs of bounded VC-dimension. Lov\'asz and Szegedy established an ``ultra strong regularity lemma'' for graphs of bounded VC-dimension \cite{LS10}. In this result, the number of parts is only $\epsilon^{-O(d^2)}$ where $d$ is the VC-dimension and almost all pairs of parts have edge density either less than $\epsilon$ or greater than $1-\epsilon$. This result was quantitatively improved by Alon, Fischer, and Newman \cite{AFN07} and extended to hypergraphs by Fox, Pach, and Suk \cite{FPS19}.

For semialgebraic hypergraphs of bounded dimension and total degree, even stronger results are known. In 2005, Alon, Pach, Pinchasi, Radoi\v{c}i\'c, and Sharir first defined semialgebraic graphs to study the crossing patterns of semialgebraic sets \cite{APPRS05}. An ``almost perfect regularity lemma'' follows from their results. Here the vertices of a semialgebraic graph are partitioned so that almost all pairs of parts are homogeneous, i.e., either complete or empty. This result was extended to hypergraphs by Fox, Gromov, Lafforgue, Naor, and Pach \cite{FGLNP12}. Both of these results were quantitatively strengthened by Fox, Pach, and Suk \cite{FPS16} who proved a polynomial ``almost perfect regularity lemma'' where the number of parts in the partition is only $\epsilon^{-C}$ for some constant $C$ depending on the dimension, uniformity, and total degree of the hypergraph.

We now formally state the definition of a semialgebraic hypergraph. We will restrict our attention to $k$-uniform hypergraphs which are also $k$-partite. This is not a serious restriction, since any hypergraph $H$ can be represented by a $k$-partite hypergraph $H'$ whose vertex set consists of $k$ copies of $V(H)$.

\begin{definition}
A \emph{semialgebraic hypergraph} in $\R^d$ is a $k$-uniform $k$-partite hypergraph $H=(P_1\sqcup\cdots\sqcup P_k,E)$ where $P_1,\ldots,P_k$ are each finite subsets of $\R^d$ such that the following holds. There are polynomials $g_1,\ldots,g_t\in\R[x_{11},\ldots,x_{kd}]$ and a function $\Phi\colon\{-1,0,1\}^t\to\{0,1\}$ so that $((x_{11},\ldots,x_{1d}),\ldots,(x_{k1},\ldots,x_{kd}))\in P_1\times\cdots\times P_k$ is an edge if and only if
\[\Phi(\sgn(g_1(x_{11},\ldots,x_{kd})),\ldots,\sgn(g_t(x_{11},\ldots,x_{kd})))=1.\]
We call $\sum_{r=1}^t \deg g_r$ the \emph{total degree} of $H$.\footnote{Previous works have used slightly different definitions of the complexity of a hypergraph. For example, \cite{FPS16} gives a semialgebraic hypergraph a complexity $(t,D)$ where $t$ is the number of polynomials used and $D$ is an upper bound on the degree of each polynomial as a function of $x_{i1},\ldots,x_{id}$ for each $i\in[k]$. Alternatively, \cite{FPSSZ17} defines the description complexity of a semialgebraic hypergraph to be $\max\{t,\deg g_1,\ldots,\deg g_t\}$. It is clear that all of these notions are equivalent up to a small loss in the parameters.}
\end{definition}

\begin{definition}
For a $k$-partite $k$-uniform hypergraph $H=(P_1\sqcup\cdots\sqcup P_k,E)$ we say that partitions $\Pi_1,\ldots,\Pi_k$ of $P_1,\ldots,P_k$ respectively form a \emph{homogeneous partition} with error $\epsilon$ if 
\[\sum\frac{|\pi_1|\cdots|\pi_k|}{|P_1|\cdots|P_k|}\leq \epsilon\]
where the sum is over $k$-tuples of parts $(\pi_1,\ldots,\pi_k)\in\Pi_1\times\cdots\times\Pi_k$ that are not homogeneous, i.e., such that $E[\pi_1\times\cdots\times \pi_k]$ is neither complete nor empty.
\end{definition}

Our first result is an optimal and oblivious regularity lemma for semialgebraic hypergraphs. We reduce the number of parts in the regularity partition to $O_{d,k}((D/\epsilon)^d)$. In other words, the power of $\epsilon$ is reduced from the inexplicit $C=C(d,k,D)$ to only the dimension $d$. In addition, the dependence on the total degree $D$ is made explicit and polynomial. We give a family of constructions where the number of parts is necessarily $\Omega_{d,k}((D/\epsilon)^d)$, proving the optimality of this result. Our regularity lemma has a second surprising property which we call obliviousness. Given the vertex set of $H$, the error parameter $\epsilon$, and the total degree $D$, our main result provides a regularity partition which is homogeneous with error $\epsilon$ for every semialgebraic hypergraph on this vertex set.

\begin{theorem}
\label{thm:main}
Fix $d,k,D\geq1$ and $\epsilon>0$. Let $P_1,\ldots,P_k\subset\R^d$ be finite sets. There exist partitions $\Pi_1,\ldots,\Pi_k$ of $P_1,\ldots,P_k$ respectively of size
\[|\Pi_1|,\ldots,|\Pi_k|=O_{d,k}\paren{\paren{D/\epsilon}^d}.\]
These have the property that for any semialgebraic hypergraph $H$ on vertex sets $P_1,\ldots,P_k$ with total degree at most $D$, the partitions $\Pi_1,\ldots,\Pi_k$ give a homogeneous partition of $H$ with error $\epsilon$.  
\end{theorem}

While this paper was in preparation, Rubin \cite{Rub24b} proved a version of this result where the number of parts is bounded by $O_{d,k,D}((1/\epsilon)^{d+o(1)})$. For a comparison of his techniques and ours, see \cref{ssec:comparison}.

A partition is called equitable if every pair of parts differs in size by at most 1. Via standard techniques one can deduce an equitable version of the main theorem with a small loss in the bound. In dimension $d=1$, our techniques give an equitable partition with no loss in the bound.

\begin{corollary}
\label{thm:main-equitable}
Fix $d,k,D\geq1$ and $\epsilon>0$. Let $P_1,\ldots,P_k\subset\R^d$ be finite sets. There exist equitable partitions $\Pi_1,\ldots,\Pi_k$ of $P_1,\ldots,P_k$ respectively of size
\[|\Pi_1|,\ldots,|\Pi_k|=O_{d,k}\paren{D^d\epsilon^{-(d+1)}}.\]
These have the property that for any semialgebraic hypergraph $H$ on vertex sets $P_1,\ldots,P_k$ with total degree at most $D$, the partitions $\Pi_1,\ldots,\Pi_k$ give a homogeneous partition of $H$ with error $\epsilon$.  
Furthermore, if $d=1$, we can instead guarantee that
\[|\Pi_1|,\ldots,|\Pi_k|=O_{k}\paren{D/\epsilon}.\]
\end{corollary}

One straightforward consequence of \cref{thm:main} is a Tur\'an-type result. This improves upon a recent result of Rubin \cite{Rub24a} who proves the same dependence on $\epsilon$ without an explicit dependence on $D$. As far as we are aware, this is the first Tur\'an result for semialgebraic graphs or hypergraphs which has an explicit dependence on the total degree.

\begin{corollary}
\label{thm:turan}
Let $H=(P_1\sqcup \cdots\sqcup P_k,E)$ be a semialgebraic hypergraph in $\R^d$ with total degree $D$ and at least $\epsilon|P_1|\cdots|P_k|$ edges. Then there exist sets $S_1\subseteq P_1,\ldots, S_k\subseteq P_k$ such that $S_1\times\cdots\times S_k\subseteq E$ and $|S_1|\cdots|S_k|=\Omega_{d,k}(D^{-d(k-1)}\epsilon^{d(k-1)+1}|P_1|\cdots|P_k|)$. In fact, for $1\leq i\leq k-1$,
\[|S_i|=\Omega_{d,k}\paren{(\epsilon/D)^d|P_i|}\]
and
\[|S_k|=\Omega\paren{\epsilon|P_k|}.\]
\end{corollary}

These results have a number of applications. For example, the regularity lemma implies the existence of very fast property testing algorithms inside the class of semialgebraic hypergraphs. The Tur\'an-type result gives better bounds on the Erd\H{o}s--Hajnal problem for semialgebraic graphs. Earlier, Alon, Pach, Pinchasi, Radoi\v{c}i\'c, and Sharir proved that semialgebraic graphs have the Erd\H{o}s--Hajnal property: every $N$-vertex semialgebraic graph contains a clique or independent set of size $N^{c}$ for some $c=c(d,D)>0$ only depending on the dimension and total degree. We improve this bound to $N^{\Omega(\log \log D/d\log D)}$ and show that this exponent is tight up to a constant factor. We will discuss these applications further in \cref{sec:applications}. In addition to these, a main application is a Zarankiewicz-type result for semialgebraic hypergraphs.

\subsection{The Zarankiewicz problem}
Perhaps the most well-known problem in extremal graph theory is the Zarankiewicz problem. This asks for the maximum number of edges in an $N$-vertex $K_{s,t}$-free graph \cite{Zar51}. The best-known upper bound is $O_t(N^{2-1/s})$ for $s\leq t$, proved by K\H{o}v\'ari, S\'os, and Tur\'an in 1954 \cite{KST54}. This bound is conjectured to be tight; despite considerable effort, matching constructions are only known when $s=2,3$, proved by Brown and Erd\H{o}s--R\'enyi--S\'os \cite{Bro66,ERS66} and when $t$ is much larger than $s$. For small values of $s$, a matching construction that works for $t\geq (s-1)!+1$ was given by Alon, Koll\'ar, R\'onyai, and Szab\'o \cite{ARS99,KRS96}. For larger values of $s$, a more recent construction of Bukh \cite{Buk24} gives a matching lower bound for $t\geq 9^{(1+o(1))s}$. All of these constructions are described as algebraic graphs defined over finite fields.

However, it is known that any such construction must have either large total degree or large dimension. This is because a $K_{u,u}$-free semialgebraic graph (over $\R$) or algebraic graph (over any field) has at most $O_{u,d,D}(N^{2-1/d})$ edges. (For $d\ll u$ this substantially improves upon the K\H{o}v\'ari--S\'os--Tur\'an bound.) This result was proved by Fox, Pach, Sheffer, Suk, and Zahl \cite{FPSSZ17} for semialgebraic graphs and recently by Milojevi\'c, Sudakov, and Tomon \cite{MST24} for algebraic graphs over arbitrary fields. Over arbitrary fields this is tight \cite{MST24}, but has been improved over the reals: Fox et al. prove that an $N$-vertex $K_{u,u}$-free semialgebraic graph in $\R^d$ with total degree $D$ has at most $O_{u,d,D}(N^{2d/(d+1)+o(1)})$ edges \cite{FPSSZ17}. Do extended this result to $K^{(k)}_{u,\ldots,u}$-free semialgebraic hypergraphs \cite{Do18}.

These semialgebraic Zarankiewicz results are quite important in discrete geometry. For example, the point-line incidence problem in $\R^2$ is naturally represented as a semialgebraic graph in $\R^2$. Since this point-line incidence graph is $K_{2,2}$-free, the above result of Fox et al. implies that $N$ points and $N$ lines in $\R^2$ form at most $O(N^{4/3})$ incidences, recovering the Szemer\'edi--Trotter theorem. In higher dimensions the point-hyperplane incidence problem becomes trivial since all $N$ points can be placed on the same line which is contained in all the hyperplanes. To make this problem non-trivial it is standard to study it under the assumption that the incidence graph is $K_{u,u}$-free. For this problem, Apfelbaum and Sharir proved that there are at most $O_{u,d}(N^{2d/(d+1)})$ incidences \cite{AS07}, improving previous works of Chazelle and Brass--Knauer \cite{Cha93,BK03}. Though there is no matching lower bound in dimensions $d\geq 3$, this remains the best-known upper bound on the point-hyperplane incidence problem.

Note in particular, the semialgebraic Zarankiewicz result of Fox et al. matches the Apfelbaum--Sharir point-hyperplane incidence bound up to a $o(1)$-loss. Our techniques imply a Zarankiewicz result for semialgebraic hypergraphs. This improves upon the bounds of Fox et al. and of Do by removing the $o(1)$-loss. In addition, we make the dependence on $u$ and the total degree of the hypergraph explicit and polynomial. As far as we are aware, this is the first result on general semialgebraic graphs which removes the $o(1)$-loss and the first which makes the dependence on the total degree explicit.

\begin{theorem}
\label{thm:zarankiewicz}
Let $H=(P_1\sqcup \cdots\sqcup P_k,E)$ be a semialgebraic hypergraph in $\R^d$ with total degree $D$. Suppose $|P_1|=\cdots=|P_k|=N$. If $H$ is $K^{(k)}_{u,\ldots,u}$-free, then the number of edges of $H$ is at most
\[|E|=O_{d,k}\paren{\sum_{s=1}^k u^{a_s}D^{b_s}N^{k-\frac{s}{(s-1)d+1}}}.\]
where
$a_s={\frac{6s+\left((s-2)(s+3)+3(k-s)(s+2)\right)(s-1)d}{6((s-1)d+1)}}$ and $b_s={\frac{(s-1)(s+2)d}{2((s-1)d+1)}}$.
For $k=2$, this implies that
\[|E|=O_d\paren{u^{\frac2{d+1}}D^{\frac{2d}{d+1}}N^{\frac{2d}{d+1}}}.\]
\end{theorem}

Later in the paper we will state and prove this result when $P_1,\ldots,P_k$ have different sizes. We also remark that this and all previous results can easily be generalized to the case when $P_1,\ldots, P_k$ live in $\R^{d_1},\ldots,\R^{d_k}$ for different dimensions $d_1,\ldots,d_k$.

In addition to removing the $o(1)$-loss in all previous applications of Fox et al.'s Zarankiewicz result, we expect this theorem to have a number of new applications due to the explicit dependence on $u$ and $D$. In \cref{sec:applications}, we give a couple immediate applications to the unit distance problem and to counting equilateral triangles.

\subsection{Proof technique and comparison to previous work}
\label{ssec:comparison}
All of our results are proved via a new polynomial method tool that we call ``multilevel polynomial partitioning''. The polynomial method was first introduced by Dvir in his proof of the finite field Kakeya conjecture \cite{Dvi09}. Guth and Katz further developed the method by introducing a tool called real polynomial partitioning as part of their breakthrough resolution of the Erd\H{o}s distinct distance problem \cite{GK15}. Real polynomial partitioning and the polynomial method more generally have revolutionized discrete geometry since their introduction fifteen years ago (see, e.g., \cite{FPSSZ17,KMS12,Kol15,Rud18,ST12,TYZ22,Wal23,Zah15,Zah19}).

Real polynomial partitioning takes a point set $P\subset\R^d$ and produces a low degree polynomial $f$ such that each connected component of $\R^d\setminus Z(f)$ contains few points of $P$. This allows one to efficiently divide a discrete geometry problem into subproblems. The main drawback of Guth and Katz's original method is that there is no control over the points of $P\cap Z(f)$. This makes it increasingly difficult to apply real polynomial partitioning as the ambient dimension of the problem increases. To overcome this issue a technique called ``constant-degree polynomial partitioning'' was developed. This technique iteratively applies Guth--Katz polynomial partitioning many times, each time with a polynomial of bounded degree. Though this method has found widespread success (see, e.g., \cite{Do18,FPSSZ17,Gut15a,ST12}), the iterative nature of the proof technique means that all applications have a $o(1)$-loss in the exponent. For example, this is the technique that Fox--Pach--Sheffer--Suk--Zahl and Do use to prove their semialgebraic Zarankiewicz results. It is a major open problem to remove the $o(1)$-loss in all applications of constant-degree polynomial partitioning.

One further extension of the polynomial method was recently developed by Walsh \cite{Wal20}. Using more input from algebraic geometry, he showed how to more efficiently apply real polynomial partitioning to a set of points $P$ lying in some variety $V$. Though his techniques are very strong they are also quite technical to apply. To date his results have only found two applications: point-hypersurface incidences \cite{Wal20} and point-curve incidences in arbitrary dimensions \cite{Wal23}.

In this paper we use Walsh's methods to prove a general-purpose tool that we call multilevel polynomial partitioning. We give two applications, using it to prove our regularity lemma and our Zarankiewicz result, though we hope that this result can be used to remove the $o(1)$-loss in many other results proved using constant-degree polynomial partitioning. Given a point set $P\subset\R^d$, multilevel polynomial partitioning constructs a partition of $P$ by applying Walsh's polynomial partitioning result at $d$ levels. In more detail, we first find a low degree polynomial $f$ that cuts $\R^d\setminus Z(f)$ into a number of open cells each with few points from $P$. For each open cell $C$ (i.e., connected component of $\R^d\setminus Z(f)$) we add $P\cap C$ as a part of our partition. It then remains to partition $P\cap Z(f)$. To handle these points we use polynomial partitioning again inside the variety $Z(f)$. We iterate this a total of $d$ times, reducing the dimension of the variety by one each iteration. To produce an efficient quantitative dependence, it is important that we use Walsh's polynomial partitioning over varieties as well as keep careful control over the number and complexity of the varieties produced. Our multilevel polynomial partitioning scheme is qualitatively similar to previous work of Matou\v{s}ek and Pat\'{a}kov\'{a} \cite{MP15} though our work has significantly stronger quantitative dependence.

To prove the semialgebraic regularity lemma, \cref{thm:main}, we apply multilevel polynomial partitioning to the vertex sets $P_1,\ldots,P_k$ with a parameter chosen in terms of $d,k,\epsilon,D$. It then remains to prove that this gives a homogeneous partition of any semialgebraic hypergraph $H$ on vertex set $P_1\sqcup\cdots\sqcup P_k$ with total degree at most $D$. To do this we use an argument that is inspired by recent work of Bukh and Vasileuski on the same-type lemma \cite{BV24}. We view our $k$ partitions of $\R^d$ as giving a single partition of $\R^{dk}$. Then a $k$-tuple of cells is non-homogeneous only if it is properly crossed by one of the defining polynomials $g_1,\ldots,g_t$ of the hypergraph $H$. We bound the number of cells crossed by the zero set of a polynomial using a Milnor--Thom-type result proved by Barone and Basu \cite{BB16}.

While this paper was in preparation, Rubin \cite{Rub24b} independently proved a version of \cref{thm:main} where the number of parts is bounded by $O_{d,k,D}((1/\epsilon)^{d+o(1)})$. His regularity lemma is also oblivious to the edge set of the hypergraph and has the additional advantage that the partition can be computed by an algorithm that runs in near linear time. However, his result has the disadvantage that the dependence of the number of parts on $D$ is not stated explicitly and is likely quite large, and the bound loses a factor of $\epsilon^{-o(1)}$. Rubin constructs his partitions via constant-degree polynomial partitioning which leads to the $o(1)$-loss. He also proves the correctness of this partition using an analysis in $\R^{dk}$ that also is inspired by the recent work of Bukh and Vasileuski \cite{BV24}.

Our Tur\'an result, \cref{thm:turan}, is a straightforward corollary of the semialgebraic regularity lemma. To prove our Zarankiewicz result, \cref{thm:zarankiewicz}, we combine the Tur\'an result with a second application of multilevel polynomial partitioning.

\subsection*{Structure of paper}

We state the multilevel polynomial partitioning scheme in \cref{sec:proof-of-regularity} and use it to prove the regularity lemma, \cref{thm:main}. In \cref{sec:proof-of-mpps} we prove the multilevel polynomial partitioning scheme. We then deduce our Tur\'an result, \cref{thm:turan}, in \cref{sec:turan} and our Zarankiewicz result, \cref{thm:zarankiewicz}, in \cref{sec:zarankiewicz}. We conclude with applications, lower bound examples for the regularity lemma, and open problems in \cref{sec:applications,sec:lower-bound,sec:open-problems} respectively.

\subsection*{Notation}

We use standard graph theory notation. We interchangeably write $A=O(B)$ and $A\lesssim B$ to mean that there exists some absolute constant $C$ so that $A\leq CB$. We write $[k]$ for $\{1,2,\ldots,k\}$.

\subsection*{Acknowledgments}
We thank Jacob Fox for numerous enlightening discussions as well as Huy Tuan Pham for helpful conversations about \cref{sec:lower-bound}. We thank Mervyn Tong for pointing out an error in an earlier version of the proof of \cref{thm:zarankiewicz-full}.

\section{Multilevel polynomial partitioning}
\label{sec:proof-of-regularity}

In this paper we use the word variety to refer to a (not-necessarily irreducible) algebraic variety in $\CC^d$. For a variety $V$ we use the typical notions of degree and dimension. We write $V(\R)$ for the real points of $V$, i.e., $V(\R)=V\cap\R^d$. Given polynomials $Q_1,\ldots,Q_r$, we write $Z(Q_1,\ldots,Q_r)$ for the variety $\{x\in\CC^d:Q_1(x)=\cdots=Q_r(x)=0\}$.

Our main result is the following multilevel polynomial partitioning scheme. The input is a finite set of points $P\subset\R^d$ and a parameter $A\geq 1$. We produce a partition of $P$ into $O_d(A^d)$ parts where each part of the partition is contained in a semialgebraic set $C$. Here $C$ is a connected component of $Z(Q_1,\ldots, Q_r)(\R)\setminus Z(f)$ where $Q_1,\ldots, Q_r,f$ are polynomials of degree $O_d(A)$.

\begin{theorem}[Multilevel polynomial partitioning]
\label{thm:main-partition}
Let $P\subset\R^d$ be a finite set of points. Given a parameter $A\geq 1$ we produce sets
\[\cV=\cV_0\sqcup\cV_1\sqcup\cdots\sqcup\cV_d.\]
For $0\leq i\leq d$, each element $V\in\cV_i$ is a tuple
\[V=(Q^{(V)}, f_V, P_V)\in\cV_i\]
where
\begin{itemize}
    \item $Q^{(V)}=(Q^{(V)}_1,\ldots,Q^{(V)}_{d-i})$ is a $(d-i)$-tuple of complex polynomials each of degree $O_d(A)$;
    \item $f_V$ is a real polynomial of degree $O_d(A)$ that vanishes identically on every irreducible component of $Z(Q^{(V)})$ of dimension at least $i+1$; and
    \item $P_V\subseteq Z(Q^{(V)})\cap P$ is a finite set of points.
\end{itemize}
These satisfy
\begin{equation}
\label{eq:total-deg}
\sum_{V\in\cV_i}\prod_{j=1}^{d-i}\deg Q^{(V)}_j\lesssim_d A^{d-i}
\end{equation}
for each $0\leq i\leq d$. Additionally
\begin{equation}
\label{eq:total-size-of-point-sets}
\sum_{V\in\cV}|P_V|\lesssim_d |P|.
\end{equation}
For each $V\in\cV$, define $\cC_V$ to be the set of connected components of $Z(Q^{(V)})(\R)\setminus Z(f_V)$. Now define the set of cells as
\[\cC=\bigsqcup_{V\in\cV}\cC_V.\]
For each $C\in\cC_V$, define $P_C=C\cap P_V$. These have the property that they form a partition
\begin{equation}
\label{eq:partition}
P=\bigsqcup_{C\in\cC}P_C
\end{equation}
and also satisfy
\begin{equation}
\label{eq:size-of-parts}
|P_C|\lesssim_d \frac{|P_V|}{A^{i}\prod_{j=1}^{d-i}\deg Q_j^{(V)}}
\end{equation}
for each $0\leq i\leq d$, each $V\in\cV_i$, and each $C\in\cC_V$.
\end{theorem}

We defer the proof until the next section. For now we demonstrate how to use multilevel polynomial partitioning by using it to prove \cref{thm:main}. This requires bounding the number of cells in the partition and bounding the number of non-homogeneous $k$-tuples of cells. We deduce this from a Milnor--Thom-type result proved by Barone and Basu \cite{BB16}. The deduction from Barone and Basu's work is somewhat technical and proceeds in a similar way to how Walsh deduces \cite[Theorem 1.5]{Wal20}.

We place Barone and Basu's result in the following form which we find more convenient to work with. For a variety $V\subseteq\CC^d$ and a point $x\in V$, we write $\cdim V$ for the maximum dimension of an irreducible component of $V$ that contains $x$.

\begin{theorem}[{\cite{BB16}, see also \cite{Wal20}}]
\label{thm:barone-basu}
Given polynomials $Q_1,\ldots, Q_{d-r}\in\CC[x_1,\ldots,x_d]$,
the number of connected components of $Z(Q_1,\ldots,Q_{d-r})(\R)$ that contain a point $x$ with $\cdim{Z(Q_1,\ldots,Q_{d-r})}=r$ is at most
\[\lesssim_d\deg Q_1\cdots \deg Q_{d-r} \left(\max_{j=1,\ldots, {d-r}}\deg Q_j\right)^{r}.\]
\end{theorem}

Note that by Krull's theorem, every $x\in Z(Q_1,\ldots,Q_{d-r})$ satisfies $\cdim{Z(Q_1,\ldots,Q_{d-r})}\geq r$ so this result counts connected components whose dimension is as small as possible. The above result follows immediately from \cite[Proposition 7.1]{Wal20} and \cite[Lemma 7.2]{Wal20}. For completeness, we include this brief deduction in the appendix.

Now we are in a position to bound the number of cells produced by the multilevel polynomial partitioning scheme.

\begin{lemma}
\label{thm:number-of-cells}
Given $A\geq 1$ and $P\subset\R^d$, let $V\in\cV_i$ and $\cC_V$ be as in \cref{thm:main-partition} for some $0\leq i\leq d$. Then
\[|\cC_V|\lesssim_d A^i\prod_{j=1}^{d-i}\deg Q_j^{(V)}.\]
Together with \cref{eq:total-deg}, this implies that $|\cC|\lesssim_d A^d$.
\end{lemma}
\begin{proof}
Consider a connected component $C$ of $Z(Q^{(V)})(\R)\backslash Z(f_V)$. Suppose $C$ is contained in $C'$, a connected component of $Z(Q^{(V)})(\R)$. If $C\neq C'$, then $Z(f_V)$ cuts $C'$ and so $C$ ``touches'' $Z(f_V)$. Thus we are in one of the three following cases:
    \begin{enumerate}[(i)]
    \item $C$ is a connected component of $Z(Q^{(V)})(\R)$; or
    \item there exists $\varepsilon_C>0$ such that $f_V$ takes all values $(0,\varepsilon_C]$ on $C$; or
    \item there exists $\varepsilon_C>0$ such that $f_V$ takes all values $[-\varepsilon_C,0)$ on $C$.
\end{enumerate}
Recall that $f_V$ vanishes on every irreducible component of $Z(Q^{(V)})$ of dimension at least $i+1$. Thus every $x\in C$ satisfies $\cdim{Z(Q^{(V)})}=i$. Therefore, by \cref{thm:barone-basu}, the number of connected components $C$ in case (i) above is at most
\[\lesssim_d\prod_{j=1}^{d-i}\deg Q_j^{(V)}\paren{\max_{j}\deg Q_j^{(V)}}^i\lesssim_d A^i\prod_{j=1}^{d-i}\deg Q_j^{(V)}.\]

To deal with the connected components in cases (ii) and (iii), pick $\varepsilon>0$ so that $\varepsilon<\varepsilon_C$ for all such $C$ and such that neither $f_V-\epsilon$ nor $f_V+\epsilon$ vanishes identically on any irreducible component of $Z(Q^{(V)})$. We can do this since there are only a finite number of $\epsilon_C$ and also a finite number of bad $\epsilon$ for which $f_V\pm\epsilon$ vanishes identically on some irreducible component of $Z(Q^{(V)})$.

Then each $C$ in case (ii) contains a connected component of $Z(Q^{(V)})(\R)\cap Z(f_V-\varepsilon)$.
Moreover, for any $x\in Z(Q^{(V)}, f_V-\varepsilon)$, we know that every irreducible component of $Z(Q^{(V)})$ which contains $x$ has dimension $i$. By our choice of $\epsilon$ we know that every irreducible component of $Z(Q^{(V)},f_V-\epsilon)$ containing $x$ has dimension $i-1$. In other words $\cdim{Z(Q^{(V)},f_V-\epsilon)}=i-1$. Therefore, applying \cref{thm:barone-basu} to $Z(Q^{(V)},f_V-\epsilon)(\R)$, we conclude that the number of connected components $C$ in case (ii) above is at most 
\[\lesssim_d\deg f_V\prod_{j=1}^{d-i}\deg Q_j^{(V)}\paren{\max\set{\deg f_V,\deg Q_1^{(V)},\ldots,\deg Q_{d-i}^{(V)}}}^{i-1}\lesssim_d A^i\prod_{j=1}^{d-i}\deg Q_j^{(V)}.\]

Applying the same argument to $Z(Q^{(V)},f_V+\epsilon)(\R)$ gives the same bound for case (iii), completing the proof.
\end{proof}

By a similar but more involved argument we can prove the following lemma which is crucial for bounding the number of non-homogeneous $k$-tuples of cells.

\begin{lemma}
\label{thm:number-of-bad-tuples}
Given finite sets $P_1,\ldots,P_k\subset\R^d$, let $\cV^{(1)},\ldots,\cV^{(k)}$ be the data produced by applying \cref{thm:main-partition} with parameter $A\geq 1$. Consider $V_1\in\cV^{(1)}_{i_1},\ldots,V_k\in\cV^{(k)}_{i_k}$ and a real polynomial $g\in\R[x_{11},\ldots,x_{kd}]$. The number of $k$-tuples of cells $(C_1,\ldots,C_k)\in\cC_{V_1}\times\cdots\times\cC_{V_k}$ such that $\sgn(g(x_{11},\ldots,x_{kd}))$ is not constant on $C_1\times\cdots\times C_k\subseteq\R^{dk}$ is at most
\[\lesssim_{d,k} \deg g\cdot A^{i_1+\cdots+i_k-1}\prod_{\ell=1}^k\prod_{j=1}^{d-i_\ell}\deg Q_j^{(V_\ell)}.\]
\end{lemma}

\begin{proof}
If $\deg g\geq A$, then using \cref{thm:number-of-cells} to bound $|\cC_{V_1}\times\cdots\times\cC_{V_k}|$ suffices to prove the desired result. Thus we assume that $\deg g\leq A$ from now on.

Define the list of polynomials
\[\tilde Q=\bigsqcup_{\ell=1}^k \set{Q_1^{(V_\ell)}(x_{\ell1},\ldots,x_{\ell d}),\ldots,Q_{d-i_\ell}^{(V_\ell)}(x_{\ell1},\ldots,x_{\ell d})}.\]
Next, define the real polynomial
\[\tilde f=\prod_{\ell=1}^k f_{V_\ell}(x_{\ell1},\ldots,x_{\ell d}).\]
Note that $Z(\tilde Q)(\R)\setminus Z(\tilde f)=\prod_{j=1}^k Z(Q^{(V_j)})(\R)\setminus Z(f_{V_j})$, so the set of connected components of $Z(\tilde Q)(\R)\setminus Z(\tilde f)$ is exactly $\set{C_1\times\cdots\times C_k:(C_1,\ldots,C_k)\in\cC_{V_1}\times\cdots\times\cC_{V_k}}$.

Since $f_{V_j}$ vanishes identically on the all irreducible components of $Z(Q^{(V_j)})$ with dimension at least $i_j+1$, we see that each point $x\in Z(\tilde Q)(\R)\setminus Z(\tilde f)$ satisfies $\cdim{Z(\tilde Q)}=i_1+\cdots+i_k$.

Define $\cC_+$ to be the set of connected components of $Z(\tilde Q)(\R)\setminus Z(\tilde f)$ on which $g$ takes both positive values and 0. Similarly define $\cC_-$ as the set of connected components on which $g$ takes both negative values and 0. Since $g\colon\R^{dk}\to\R$ is continuous in the Euclidean topology, it is clear that if $\sgn(g(x_{11},\ldots,x_{kd}))$ is not constant on a connected component $C$, then $C\in\cC_+\cup\cC_-$. We now bound $|\cC_+|$; bounding $|\cC_-|$ is symmetric.

Since $Z(\tilde Q)$ has finitely many irreducible components, there are a finite number of bad $\delta\in\R$ such that $g-\delta$ vanishes identically on an irreducible component of $Z(\tilde Q)$. For each $C\in\cC_+$, there exists $\delta_C>0$ so that $g$ takes all values $[0,\delta_C]$ on $C$. Now pick $\delta>0$ so that $\delta<\delta_C$ for all $C\in\cC_+$ and so that $g-\delta$ does not vanish identically on any irreducible component of $Z(\tilde Q)$. Clearly each $C\in\cC_+$ contains a connected component of $Z(\tilde Q,g-\delta)(\R)\setminus Z(\tilde f)$.

Furthermore, since each point $x\in Z(\tilde Q)(\R)\setminus Z(\tilde f)$ satisfies $\cdim{Z(\tilde Q)}=i_1+\cdots+i_k$, by our choice of $\delta$, we see that $\cdim{Z(\tilde Q,g-\delta)}=i_1+\cdots+i_k-1$ for all $x\in Z(\tilde Q,g-\delta)(\R)\setminus Z(\tilde f)$.

As in the proof of \cref{thm:number-of-cells}, each connected component $C$ of $Z(\tilde Q,g-\delta)(\R)\setminus Z(\tilde f)$ satisfies one of the three following properties:
\begin{enumerate}[(i)]
    \item $C$ is a connected component of $Z(\tilde Q,g-\delta)(\R)$; or
    \item there exists $\epsilon_C>0$ such that $\tilde f$ takes all values $(0,\epsilon_C]$ on $C$; or
    \item there exists $\epsilon_C>0$ such that $\tilde f$ takes all values $[-\epsilon_C,0)$ on $C$.
\end{enumerate}
Pick $\epsilon>0$ such that $\epsilon<\epsilon_C$ for all such $C$ and such that neither $\tilde f-\epsilon$ nor $\tilde f+\epsilon$ vanishes identically on any irreducible component of $Z(\tilde Q,g-\delta)$. Then we see that each point $x\in Z(\tilde Q,g-\delta,\tilde f\pm\epsilon)(\R)\setminus Z(\tilde f)$ satisfies $\cdim{Z(\tilde Q,g-\delta,\tilde f\pm\epsilon)}=i_1+\cdots+i_k-2$.

Thus by \cref{thm:barone-basu}, the number of $C$ in case (i) is bounded by
\[\begin{split}\lesssim_{d,k}\deg g\prod_{\ell=1}^k\prod_{j=1}^{d-i_\ell}\deg Q_j^{(V_\ell)}&\paren{\max\set{\deg g,\deg Q_1^{(V_1)},\ldots,\deg Q_{d-i_k}^{(V_k)}}}^{i_1+\cdots+i_k-1}\\
&\lesssim_{d,k} \deg g\cdot A^{i_1+\cdots+i_k-1}\prod_{\ell=1}^k\prod_{j=1}^{d-i_\ell}\deg Q_j^{(V_\ell)}.
\end{split}\]

Furthermore, each $C$ in case (ii) above contains a connected component of $Z(\tilde Q,g-\delta,\tilde f-\epsilon)(\R)$ and each $C$ in case (iii) above contains a connected component of $Z(\tilde Q,g-\delta,\tilde f+\epsilon)(\R)$. Again by \cref{thm:barone-basu}, the number of $C$ in cases (ii) and (iii) is bounded by
\[\begin{split}
\lesssim_{d,k}\deg g\cdot\deg \tilde f\prod_{\ell=1}^k\prod_{j=1}^{d-i_\ell}\deg Q_j^{(V_\ell)}&\paren{\max\set{\deg g,\deg\tilde f,\deg Q_1^{(V_1)},\ldots,\deg Q_{d-i_k}^{(V_k)}}}^{i_1+\cdots+i_k-2}\\
&\lesssim_{d,k} \deg g\cdot A^{i_1+\cdots+i_k-1}\prod_{\ell=1}^k\prod_{j=1}^{d-i_\ell}\deg Q_j^{(V_\ell)}.
\end{split}\]

Summing these three cases gives the desired bound on $|\cC_+|$. An identical argument bounds $|\cC_-|$.
\end{proof}

\begin{proof}[Proof of \cref{thm:main}]
For some large constant $C_{d,k}$, define $A=C_{d,k}D/\epsilon\geq 1$. Let $\cV^{(1)},\ldots,\cV^{(k)}$ be the data produced by applying \cref{thm:main-partition} to $P_1,\ldots, P_k$ respectively with parameter $A$. Let $\Pi_1,\ldots,\Pi_k$ be the partitions of $P_1,\ldots, P_k$ produced. Then by \cref{thm:number-of-cells} and \cref{eq:total-deg},
\[|\Pi_\ell|=|\cC^{(\ell)}|=\sum_{i=0}^d\sum_{V\in\cV_i^{(\ell)}}|\cC_V|\lesssim_d\sum_{i=0}^d\sum_{V\in\cV_i^{(\ell)}}A^i\prod_{j=1}^{d-i}\deg Q_j^{(V)}\lesssim_d\sum_{i=0}^dA^d\lesssim A^d\]
for each $1\leq \ell\leq k$.

Now let $H$ be any semialgebraic hypergraph on $P_1,\ldots,P_k$ defined by polynomials $g_1,\ldots,g_t$ with $\sum_r\deg g_r\leq D$. By \cref{thm:number-of-bad-tuples}, for each $g_r$ and each $V_1\in\cV_{i_1}^{(1)},\ldots,V_k\in\cV_{i_k}^{(k)}$, the number of $k$-tuples of cells in $\cC_{V_1}\times\cdots\times\cC_{V_k}$ on which $\sgn(g_r(x_{11},\ldots,x_{kd}))$ is not constant is bounded by
\[\lesssim_{d,k} \deg g_r\cdot A^{i_1+\cdots+i_k-1}\prod_{\ell=1}^k\prod_{j=1}^{d-i_\ell}\deg Q_j^{(V_\ell)}.\]
By \cref{eq:size-of-parts}, the fraction of $k$-tuples of vertices contributed by each $k$-tuple of cells $(C_1,\ldots,C_k)\in\cC_{V_1}\times\cdots\times\cC_{V_k}$ is
\[\prod_{\ell=1}^k\frac{|P_{C_\ell}|}{|P_\ell|}\lesssim_{d,k}\prod_{\ell=1}^k\frac{|P_{V_\ell}|}{A^{i_\ell}|P_\ell|\prod_{j=1}^{d-i_\ell}\deg Q_j^{(V_\ell)}}.\]
Summing the contributions of all $k$-tuple of cells with this particular choice of $r$ and $V_1,\ldots,V_k$, we get a contribution of
\[\lesssim_{d,k}\frac{\deg g_r}{A}\prod_{\ell=1}^{k}\frac{|P_{V_\ell}|}{|P_\ell|}.\]
By \cref{eq:total-size-of-point-sets}, summing over all choices of $V_1,\ldots, V_k$ gives a contribution of
\[\lesssim_{d,k}\frac{\deg g_r}{A}.\]
Finally, summing over the choices of $r$, we see that at most a
\[\lesssim_{d,k}\sum_{r=1}^t\frac{\deg g_r}{A}\leq\frac{D}{A}=\frac{\epsilon}{C_{d,k}}\]
fraction of $k$-tuples of vertices lie in a $k$-tuple of cells on which $\sgn(g_r(x_{11},\ldots,x_{kd}))$ is not constant for some $r$. Choosing $C_{d,k}$ large enough makes the fraction at most $\epsilon$, which means that $\Pi_1,\ldots,\Pi_k$ form a homogeneous partition of $H$ with error $\epsilon$.
\end{proof}

\begin{proof}[Proof of \cref{thm:main-equitable}]
The deduction of the $d\geq2$ case of \cref{thm:main-equitable} from \cref{thm:main} is standard (see, e.g., \cite[Proof of Theorem 1.3]{FPS16}). We include a proof for completeness.

Let $\Pi_1,\ldots, \Pi_k$ be the partitions of $P_1,\ldots,P_k$ produced by \cref{thm:main} with parameters $d,k,D\geq1$ and $\epsilon/2$. For each $1\leq i\leq k$, say say $N_i=|P_i|$ and define $K_i=\ceil{4k\epsilon^{-1}|\Pi_i|}$. Index $P_i=\{p_1,\ldots,p_{N_i}\}$ so that each part $\pi\in\Pi_i$ is an interval in this indexing. Now define $\Pi_i'$ to be an equitable partition of $P_i$ into $K_i$ pieces, each of which is an interval of length $\floor{N_i/K_i}$ or $\ceil{N_i/K_i}$. We claim that $\Pi_1',\ldots,\Pi_k'$ have the desired properties.

Call a part $\pi\in\Pi_i'$ bad if it is not contained in a part of $\Pi_i$. Note that the number of bad parts is at most $|\Pi_i|$ and each bad part has size at most $2N_i/K_i$. For a semialgebraic hypergraph $H$ on $P_1\sqcup\cdots\sqcup P_k$, a $k$-tuple of vertices $(x_1,\ldots,x_k)$ is contained in a $k$-tuple of parts of $(\pi_1,\ldots,\pi_k)\in \Pi_1'\times\cdots\times\Pi_k'$ that is not homogeneous only if either it is contained in a $k$-tuple of parts of $\Pi_1\times\cdots\times\Pi_k$ that is not homogeneous or one of the $\pi_i$ is bad. The number of $k$-tuples of vertices of the first type is at most $(\epsilon/2)N_1\cdots N_k$. The number of $k$-tuples of vertices of the second type is at most
\[\sum_{i=1}^k\paren{\prod_{j\neq i} N_j}\frac{2N_i}{K_i}|\Pi_i|= N_1\cdots N_k\sum_{i=1}^k \frac{2|\Pi_i|}{K_i}\leq \frac{\epsilon}{2}N_1\cdots N_k.\]

Therefore $\Pi_1',\ldots,\Pi_k'$ form a homogeneous partition of $H$ with error $\epsilon$. Note that
\[|\Pi_i'|=K_i=\ceil{4k\epsilon^{-1}|\Pi_i|}=O_{d,k}(D^d\epsilon^{-(d+1)}).\]

Now we show how to save a factor of $\epsilon$ when $d=1$. For $P_i\subset\R$ and any positive integer $A$, we can find $y_1,\ldots,y_{A-1}$ so that
\[P_i\cap(-\infty,y_1), P_i\cap(y_1,y_2),\ldots, P_i\cap(y_{A-2},y_{A-1}), P_i\cap(y_{A-1},+\infty)\]
is an equitable partition of $P_i$ into $A$ pieces. Instead of applying \cref{thm:main-partition}, we will construct the data that the theorem produces by hand. Define $\cV_0=\emptyset$ and $\cV_1=\set{(Q^{(\CC)},f_{\CC},P_{\CC})}$ where $Q^{(\CC)}=\emptyset$, where $P_{\CC}=P_i$, and where $f_{\CC}(x)=\prod_{j=1}^{A-1}(x-y_j)$. One can easily check that these satisfy all of the desired properties and that the partition of $P_i$ produced is the above equitable partition.

Applying this procedure for $i=1,\ldots,k$ with $A=C_{1,k}D/\epsilon$, we produce equitable partitions $\Pi_1,\ldots,\Pi_k$ of $P_1,\ldots, P_k$. Then the proof of \cref{thm:main} shows that these partitions have the desired property.
\end{proof}

\section{Constructing the partition}
\label{sec:proof-of-mpps}

In this section, we explain how to construct the multilevel polynomial partition in \cref{thm:main-partition}. The key tool is a quantitatively efficient real polynomial partitioning result over irreducible varieties introduced by Walsh \cite{Wal20}. This allows us to afford to do real polynomial partitioning with polynomials of super-constant degree and lets us avoid the $o(1)$-loss in the exponent. These tools were also used in a later paper of Walsh \cite{Wal23} to prove an optimal point-curve incidence bound, though that setting is considerably more technical than this one.

To introduce Walsh's real polynomial partitioning result, we need to introduce a more refined notion of degree that gives more precise quantitative control over a variety. This will be used extensively throughout the proof.

\begin{definition}[Partial degrees]
For an irreducible variety $V\subseteq\CC^d$ and $1\leq i \leq d-\dim V$, define the \emph{$i$-th partial degree}, $\delta_i(V)$, to be the minimum positive integer such that there exist polynomials $f_1,\ldots,f_t\in\CC[x_1,\ldots,x_d]$ of degree at most $\delta_i(V)$ such that $V\subseteq Z(f_1,\ldots,f_t)$ and such that every irreducible component of $Z(f_1,\ldots,f_t)$ containing $V$ has dimension at least $d-i$. We define $\delta_0(V)=0$ and $\delta_i(V)=\infty$ for $i>d-\dim V$.
\end{definition}

By a short argument, \cite[Lemma 2.3]{Wal20}, these are known to satisfy $\delta_1(V)\leq\delta_2(V)\leq\cdots\leq\delta_{d-\dim V}(V)$.

To give some intuition for this definition, consider polynomials $Q_1,\ldots, Q_{d-r}\in\CC[x_1,\ldots x_d]$ satisfying $\deg Q_1\leq \cdots\leq\deg Q_{d-r}$. Suppose that the variety $V=Z(Q_1,\ldots, Q_{d-r})$ is irreducible. If $Z(Q_1),\ldots,Z(Q_{d-r})$ intersect nicely enough, (i.e., if $V$ is a complete intersection) then $V$ will have dimension $r$ and degree $\deg Q_1\cdots\deg Q_{d-r}$. Furthermore, in this case $\delta_i(V) = \deg Q_i$ for all $i$. Although not every irreducible variety is a complete intersection, Walsh's techniques say that we are always approximately in this situation. In particular, for any irreducible variety $V\subseteq\CC^d$ of dimension $r$, there exist polynomials $Q_1,\ldots,Q_{d-r}\in\CC[x_1,\ldots, x_d]$ so that $V$ is an irreducible component of $Z(Q_1,\ldots, Q_{d-r})$ and so that $\delta_i(V)=\deg Q_i$. Furthermore, $\deg V$ and $\deg Q_1\cdots\deg Q_{d-r}$ will be equal up to a multiplicative constant that only depends on the ambient dimension. We state the second part of this result below and something a little stronger than the first result as \cref{thm:envelope-excising}.

\begin{proposition}[{\cite[Corollaries 5.4 and 5.7]{Wal20}}]
\label{thm:inverse-bezout}
For an irreducible variety $V\subseteq\CC^d$,
\[\prod_{i=1}^{d-\dim V}\delta_i(V)\lesssim_d \deg V\leq\prod_{i=1}^{d-\dim V}\delta_i(V).\]
\end{proposition}

Having introduced the definition of partial degrees, we can now state the result of Walsh that allows us to do efficient real polynomial partitioning over any irreducible varieties. The formulation below is slightly different from what is stated in his paper. We include a short deduction in the appendix.

\begin{theorem}[{\cite[Theorem 3.2]{Wal20}}]
\label{thm:real-poly-part}
Let $V\subseteq\CC^d$ be an irreducible variety and let $P\subseteq V(\R)$ be a finite set of points. For $A\geq 1$, there exists a polynomial $f\in\R[x_1,\ldots,x_d]$ with $\deg f\lesssim_d A$ so that $f$ does not vanish identically on $V$ and each connected component of $\R^d\setminus Z(f)$ contains at most
\[\lesssim_d\frac{|P|}{A^{i}\prod_{j=1}^{d-i}\delta_j(V)}\]
points of $P$, where $\dim V\leq i\leq d$ is the unique integer with $\delta_{d-i}(V)\leq A<\delta_{d-i+1}(V)$.
\end{theorem}

As mentioned above, we can approximate any irreducible variety $V\subseteq \CC^d$ by $Z(Q_1,\ldots, Q_{d-\dim V})$ for some polynomials  $Q_1,\ldots, Q_{d-\dim V}$ with $\deg Q_i = \delta_i(V)$ (see \cite[Lemma 5.3]{Wal20}). However, for technical reasons, we will need to make a slightly different choice for the polynomials $Q_1,\ldots, Q_{d-\dim V}$. To be more specific, the other irreducible components of $Z(Q_1,\ldots, Q_{d-\dim V})$ might have dimensions larger than expected, making it difficult to apply \cref{thm:barone-basu}. Below we make a definition to identify those problematic components and then introduce another result of Walsh that overcomes this issue.

\begin{definition}
For an irreducible variety $V\subseteq\CC^d$ of dimension $r$ and polynomials $Q_1,\ldots,Q_{d-r}$ such that $V$ is an irreducible component of $Z(Q_1,\ldots,Q_{d-r})$, we define the \emph{envelope} $\cE_V(Q_1,\ldots,Q_{d-r})$ by
\[\cE_V(Q_1,\ldots,Q_{d-r})=\bigcup_{i=1}^{d-r}\cE_V^{(i)}(Q_1,\ldots,Q_{d-r})\]
where $\cE_V^{(i)}(Q_1,\ldots,Q_{d-r})$ is the union of the irreducible components of $Z(Q_1,\ldots,Q_{i})$ of dimension at least $d-i+1$.
\end{definition}

By Krull's theorem, it is clear that $\dim_x^\CC(Z(Q_1,\ldots,Q_{d-r}))=r$ for all $x\in Z(Q_1,\ldots,Q_{d-r})\setminus \cE_V(Q_1,\ldots,Q_{d-r})$. Therefore to make sure that we are in position to apply \cref{thm:barone-basu}, we should manually excise $\cE_V(Q_1,\ldots,Q_{d-r})$ via a polynomial of small degree. This was done in \cite[Section 6]{Wal20}. Expanding the definitions used in \cite[Lemma 6.8]{Wal20} gives the following statement. (Actually it gives something strictly stronger, but we only state the parts that we will use.)

\begin{proposition}[{\cite[Lemma 6.8]{Wal20}}]
\label{thm:envelope-excising}
Given an irreducible variety $V\subseteq\CC^d$ of dimension $j$, there exist polynomials $Q_1,\ldots,Q_{d-j}$ and $F_1,\ldots,F_{d-j-1}$ with the following properties:
\begin{itemize}
    \item $\deg Q_k\lesssim_d \delta_k(V)$ for each $1\leq k\leq d-j$;
    \item $\deg F_k<\delta_k(V)$ for each $1\leq k\leq d-j-1$;
    \item $V\subseteq Z(Q_1,\ldots,Q_{d-j})$;
    \item $F_k$ vanishes identically on each irreducible component of $\cE_V(Q_1,\ldots,Q_{d-j})$ of dimension $d-k$ for each $1\leq k\leq d-j-1$;
    \item $F_k$ does not vanish identically on $V$ for each $1\leq k\leq d-j-1$.
\end{itemize}
\end{proposition}

We now prove \cref{thm:main-partition} using \cref{thm:real-poly-part} and \cref{thm:envelope-excising}.

\begin{proof}[Proof of \cref{thm:main-partition}]
We construct the desired data in steps numbered $d,d-1,\ldots,0$. At step $j$ we start with a set $\cU_j$ whose elements are pairs $(V,P_V)$ where $V\subseteq\CC^d$ is an irreducible $j$-dimensional variety and $P_V\subseteq V\cap P$. During this step we will add some elements to $\cV_j,\cV_{j+1},\ldots,\cV_d$ and also produce a set $\cU_{j-1}$.

To start, define $\cU_d=\{(\CC^d,P)\}$. We will ensure that before step $j$, we have the partition
\begin{equation}
\label{eq:partition-inductive-statement}
P=\bigsqcup_{(V,P_V)\in\cU_j}P_V\sqcup\bigsqcup_{V\in\cV}\bigsqcup_{C\in\cC_V}P_C.
\end{equation}
Note that this is satisfied before step $d$.

Suppose we are given $\cU_j$ for some $0\leq j\leq d$ and that \cref{eq:partition-inductive-statement} is satisfied. Step $j$ proceeds as follows. For each $(V,P_V)\in\cU_j$, define $j\leq i\leq d$ so that 
\[\delta_1(V)\leq\delta_2(V)\leq\cdots\leq\delta_{d-i}(V)\leq A<\delta_{d-i+1}(V)\leq\cdots\leq\delta_{d-j}(V).\]
We apply \cref{thm:envelope-excising} to $V$ to produce polynomials $Q_1,\ldots,Q_{d-j}$ and $F_1,\ldots,F_{d-j-1}$.

By real polynomial partitioning, \cref{thm:real-poly-part}, there exists a real polynomial $f_{\prt}$ of degree $O_d(A)$ that does not vanish identically on $V$ such that each connected component of $\R^d\setminus Z(f_{\prt})$ contains at most
\[\lesssim_d\frac{\abs{P_V}}{A^{i}\prod_{\ell=1}^{d-i}\delta_\ell(V)}\]
points of $P_V$. Define $f_V=|F_1F_2\cdots F_{d-i-1}|^2f_{\prt}$. We add $V=(Q^{(V)},f_V,P_V)$ to $\cV_i$ where $Q^{(V)}=(Q_1,\ldots,Q_{d-i})$. Note that this is a $(d-i)$-tuple of polynomials; the last $i-j$ polynomials are not included in the data.

By our choice of $i$ and the above properties, we see that each of the polynomials $Q_1,\ldots,Q_{d-i},f_V$ has degree $O_d(A)$. Clearly $f_V$ is a real polynomial. By definition, every irreducible component of $Z(Q_1,\ldots,Q_{d-i})$ of dimension at least $i+1$ is contained in an irreducible component of $\cE_V(Q_1,\ldots,Q_{d-j})$ of dimension at least $i+1$, meaning that $f_V$ vanishes identically on it. Finally, $P_V\subseteq V\subseteq Z(Q_1,\ldots,Q_{d-j})\subseteq Z(Q_1,\ldots,Q_{d-i})$. This shows that $(Q^{(V)},f_V,P_V)$ has the desired properties.

We define $\cC_V$ to be the set of connected components of $Z(Q^{(V)})(\R)\setminus Z(f_V)$ and for each $C\in\cC_V$ we define $P_C=C\cap P_V$. Note that $\{P_C\}_{C\in\cC_V}$ is a partition of $P_V\setminus Z(f_V)$.

Let $W_1,\ldots,W_s$ be the irreducible components of $V\cap Z(f_V)$. Note that neither $f_{\prt}$ nor any of $F_1,\ldots,F_{d-i-1}$ vanish identically on $V$, meaning that $f_V$ does not vanish identically on $V$. Thus by Krull's theorem, each $W_k$ has dimension $j-1$.

We define $P_{W_k}$ so that $P_{W_k}\subseteq W_k\cap P_V$ and so that they partition $P_V\cap Z(f_V)$. This is clearly possible since $P_V\cap Z(f_V)\subset V\cap Z(f_V)=\bigcup_{k=1}^s W_k$. (There may be multiple ways to define these sets if some point lies in multiple of the $W_k$ but we make these choices arbitrarily.) We place $(W_1,P_{W_1}),\ldots,(W_s,P_{W_s})$ into $\cU_{j-1}$. Note that since $\{P_C\}_{C\in\cC_V}$ partitions $P_V\setminus Z(f_V)$ and $\{P_{W_k}\}_{k\in[s]}$ partitions $P_V\cap Z(f_V)$, we see that \cref{eq:partition-inductive-statement} is still satisfied. Doing this procedure for each $(V,\cP_V)\in\cV_j$ completes step $j$.

We run this process for steps $j=d,d-1,\ldots,0$. At the end we have produced the desired data. For $V\in\cV_i$, we chose $f_{\prt}$ so that each connected component of $\R^d\setminus Z(f_{\prt})$ contains at most 
\[\lesssim_d\frac{\abs{P_V}}{A^{i}\prod_{j=1}^{d-i} \delta_j(V)}\lesssim_d\frac{\abs{P_V}}{A^{i}\prod_{j=1}^{d-i}\deg Q_j^{(V)}}\] points of $P_V$. Since each connected components of $Z(Q^{(V)})(\R)\setminus Z(f_V)$ is contained in one of these connected components, we see that \cref{eq:size-of-parts} is satisfied. By \cref{eq:partition-inductive-statement}, we conclude that \cref{eq:partition} is satisfied. Next note that for each $0\leq j\leq d$, the sets $\{P_V\}_{V\in\cU_j}$ are disjoint, implying \cref{eq:total-size-of-point-sets} since $\sum_{V\in\cV}|P_V|\leq\sum_{j=0}^d\sum_{(V,P_V)\in\cU_j}|P_V|\leq (d+1)|P|$. Finally, since $\deg f_V\lesssim_d A$ for each $V$, by B\'ezout's theorem, we see that for each $0\leq j\leq d$,
\[\sum_{(V,P_V)\in\cU_j}\deg V\lesssim_d A^{d-j}.\]
By \cref{thm:inverse-bezout}, we have
\[\deg V\gtrsim_d \delta_1(V)\cdots\delta_{d-j}(V)\gtrsim_d \delta_1(V)\cdots\delta_{d-i}(V)A^{j-i}\gtrsim_d \deg Q_1^{(V)}\cdots\deg Q_{d-i}^{(V)}A^{j-i}\]
for every $(V,P_V)\in \cU_j$ with $i$ satisfying $\delta_{d-i}(V)\leq A<\delta_{d-i+1}(V)$.
Combining the last two inequalities implies \cref{eq:total-deg}, proving the desired result.
\end{proof}

\section{Tur\'an results}
\label{sec:turan}

In this section we prove \cref{thm:turan}. We deduce this from the following stronger result.

\begin{proposition}
\label{thm:generalized-turan}
Let $H=(P_1\sqcup \cdots\sqcup P_k,E)$ be a semialgebraic hypergraph in $\R^d$ with total degree $D$ and at least $\epsilon|P_1|\cdots|P_k|$ edges. Fix $1\leq \ell\leq k$. There exist sets $S_1\subseteq P_1,\ldots, S_{\ell}\subseteq P_{\ell}$ and $T\subseteq P_{\ell+1}\times\cdots\times P_k$ such that $S_1\times\cdots\times S_\ell\times T\subseteq E$ and such that for $1\leq i\leq \ell$,
\[|S_i|=\Omega_{d,k}\paren{(\epsilon/D)^d|P_i|}\]
and
\[|T|\geq \frac{\epsilon}3 |P_{\ell+1}|\cdots|P_k|.\]
\end{proposition}

The $\ell=k-1$ of this result immediately implies \cref{thm:turan}. In \cref{sec:zarankiewicz} we will use the $\ell=1$ case in the derivation of \cref{thm:zarankiewicz}.

\begin{proof}
First we handle the $d\geq 2$ case. Fix $d,k,\ell$ and let $C=C(d,k)$ be a constant chosen sufficiently large in terms of the constant in \cref{thm:main}. We will prove by induction on $|P_1|+\cdots+|P_\ell|$ that we can find the desired sets $S_1,\ldots, S_\ell,T$ satisfying
\[|S_i|\geq\frac1{C}\paren{\frac\epsilon D}^d|P_i|\]
for $1\leq i\leq \ell$ and
\[|T|\geq \frac{\epsilon}3 |P_{\ell+1}|\cdots|P_k|.\]

The base case $\abs{P_1}+\cdots +\abs{P_{\ell}}=0$ is trivial. 
For the inductive step, we apply \cref{thm:main} to $H$ with parameter $\varepsilon/3$ to produce partitions $\Pi_1,\ldots, \Pi_k$. For each $1\leq i\leq \ell$, let $B_i$ be the union of parts $\pi\in\Pi_i$ with $\abs{\pi}\leq \frac1{9k^2}\frac{|P_i|}{|\Pi_i|}$, i.e., $B_i$ is the union of the parts that are much smaller than average.

We split into two cases. If at most $(\varepsilon/3)\abs{P_1}\cdots \abs{P_k}$ edges involve a vertex in $B_1\cup \cdots \cup B_{\ell}$, define $H'=(P_1\sqcup\cdots\sqcup P_k,E')$ to be the subhypergraph of $H$ formed by removing those edges along with the edges contained in an inhomogeneous $k$-tuple of parts.
Then $|E'|\geq (\varepsilon/3)\abs{P_1}\cdots \abs{P_k}$. By the pigeonhole principle, there exists some $\ell$-tuple $(x_1,\ldots,x_\ell)\in P_1\times\cdots\times P_\ell$ such that 
\[T=\{(x_{\ell+1},\ldots,x_k)\in P_{\ell+1}\times\cdots\times P_k:(x_1,\ldots,x_k)\in E'\}\]
has size at least $(\epsilon/3)|P_{\ell+1}|\cdots|P_k|$. For $1\leq i\leq \ell$, define $S_i\in\Pi_i$ to be the part of the partition containing $x_i$. We see that $S_1\times\cdots\times S_\ell\times T\subseteq E'\subseteq E$ and for $1\leq i\leq \ell$,
\[|S_i|\geq \frac1{9k^2}\frac{|P_i|}{|\Pi_i|}\gtrsim_{d,k}\paren{\frac\epsilon D}^d|P_i|.\]
Choosing $C$ sufficiently large, this completes the induction in this case.

In the remaining case, there must be some $1\leq i\leq \ell$ such that the number of edges involving $B_i$ is at least $(\varepsilon/3k)\abs{P_1}\cdots \abs{P_k}$. Define $H'=(P_1\sqcup\cdots\sqcup B_i\sqcup\cdots\sqcup P_k,E')$ to be the induced subhypergraph of $H$ on vertex set $P_1\sqcup\cdots\sqcup B_i\sqcup\cdots\sqcup P_k$. Note that
\[|B_i|\leq\frac{1}{9k^2}\frac{|P_i|}{|\Pi_i|}|\Pi_i|=\frac{|P_i|}{9k^2}.\]
Furthermore, the edge density of $H'$ is
\[\eta=\frac{|E'|}{|P_1|\cdots|B_i|\cdots|P_k|}\geq \frac{\frac\epsilon{3k}|P_1|\cdots|P_k|}{|P_1|\cdots|B_i|\cdots|P_k|}= \frac{\varepsilon}{3k}\cdot\frac{|P_i|}{|B_i|}\geq 3k\epsilon>\epsilon.\]
Since $|P_1|+\cdots+|B_i|+\cdots+|P_k|<|P_1|+\cdots+|P_k|$, we can apply the inductive hypothesis to $H'$. This produces $S_1\subseteq P_1,\ldots,S_i\subseteq B_i\subset P_i,\ldots,S_\ell\subseteq P_\ell$ and $T\subseteq P_{\ell+1}\times\cdots\times P_k$ such that $S_1\times\cdots\times S_\ell\times T\subseteq E'\subseteq E$ and such that
\[|S_j|\geq\frac1C\paren{\frac\eta D}^d|P_j|\geq\frac1C\paren{\frac\epsilon D}^d|P_j|\]
for each $1\leq j\leq \ell$ with $j\neq i$. Additionally, 
\[\begin{split}
|S_i|\geq\frac1C\paren{\frac\eta D}^d|B_i|&\geq\frac1C\paren{\frac{\frac{\epsilon}{3k}\cdot\frac{|P_i|}{|B_i|}} D}^d|B_i|=\frac1C\paren{\frac\epsilon D}^d\frac{\paren{\frac{|P_i|}{|B_i|}}^{d-1}}{(3k)^d}|P_i|\\&\geq\frac1C\paren{\frac\epsilon D}^d\frac{(9k^2)^{d-1}}{(3k)^d}|P_i|\geq\frac1C\paren{\frac\epsilon D}^d|P_i|
\end{split}\]
where the last inequality holds since $d\geq 2$. Finally the inductive hypothesis also gives
\[|T|\geq\frac\eta3|P_{\ell+1}|\cdots|P_k|\geq\frac\epsilon3|P_{\ell+1}|\cdots|P_k|.\]
This completes the induction.

Now we prove the result for $d=1$. Let $\Pi_1,\ldots,\Pi_k$ be the equitable partitions formed by applying \cref{thm:main-equitable} with error parameter $\epsilon/2$. Note that since $d=1$, we have $|\Pi_i|=O_k(D/\epsilon)$. Let $H'=(P_1\sqcup\cdots\sqcup P_k,E')$ be the subhypergraph of $H$ formed by removing the edges contained in an inhomogeneous $k$-tuple of parts. Clearly $|E'|\geq (\varepsilon/2)\abs{P_1}\cdots \abs{P_k}$. By the pigeonhole principle, there exists some $\ell$-tuple $(x_1,\ldots,x_\ell)\in P_1\times\cdots\times P_\ell$ such that 
\[T=\{(x_{\ell+1},\ldots,x_k)\in P_{\ell+1}\times\cdots\times P_k:(x_1,\ldots,x_k)\in E'\}\]
has size at least $(\epsilon/2)|P_{\ell+1}|\cdots|P_k|$. For $1\leq i\leq \ell$, define $S_i\in\Pi_i$ to be the part of the partition containing $x_i$. We see that $S_1\times\cdots\times S_\ell\times T\subseteq E'\subseteq E$ and that $S_1,\ldots,S_\ell,T$ are of the appropriate sizes.
\end{proof}

\section{The Zarankiewicz problem}
\label{sec:zarankiewicz}

We prove the graph case ($k=2$) of the Zarankiewicz-type result first. The proof is essentially a specialization of the general proof, but the calculations are significantly simpler so we think it is valuable to read this case first.

The graph case is proved in two steps. First we deduce a weak Zarankiewicz result from our Tur\'an result, \cref{thm:turan}. Then we upgrade the weak Zarankiewicz result to the strong one via multilevel polynomial partitioning, \cref{thm:main-partition}. These are the same two steps used by Fox, Pach, Sheffer, Suk, and Zahl in the proof of their semialgebraic Zarankiewicz result. They prove a similar weak Zarankiewicz result \cite[Theorem 2.1]{FPSSZ17} using VC-dimension techniques and then upgrade it using constant-degree polynomial partitioning. Note in contrast, both halves of our proof use only polynomial partitioning and we avoid the $o(1)$-loss by using polynomial partitioning using polynomials of fairly high degree.

One should compare the following result to the K\H{o}v\'ari--S\'os--Tur\'an theorem which gives the bound of $O_u(N^{2-1/u})$ edges. Here we improve this to $O_{u,d}(N^{2-1/d})$ which is typically much smaller.

\begin{corollary}
\label{thm:zarankiewicz-weak}
Let $G=(P\sqcup Q,E)$ be a semialgebraic graph in $\R^d$ with total degree $D$. If $G$ is $K_{u,u}$-free, then the number of edges of $G$ is at most
\[|E|=O_{d}\paren{u^{1/d}D|P||Q|^{1-1/d}+u|Q|}.\]
\end{corollary}

\begin{proof}
Let $|E|=\epsilon|P||Q|$. By \cref{thm:turan}, $G$ contains $T\times S$ with 
\[|S|\gtrsim_{d}(\epsilon/D)^d|Q|\qquad\text{and}\qquad |T|\gtrsim\epsilon|P|.\]
Since $G$ is $K_{u,u}$-free we conclude that one of these sets is smaller than $u$. Thus either
\[(\epsilon/D)^d|Q|\lesssim_du\qquad\text{or}\qquad\epsilon|P|\lesssim u.\]
These inequalities rearrange to
\[|E|=\epsilon|P||Q|\lesssim_d{D\paren{\frac{u}{|Q|}}^{1/d}|P||Q|}\qquad\text{or}\qquad|E|\lesssim u|Q|.\qedhere\]
\end{proof}

\begin{theorem}
\label{thm:zarankiewicz-2-unif}
Let $G=(P\sqcup Q,E)$ be a semialgebraic graph in $\R^d$ with total degree $D$. If $G$ is $K_{u,u}$-free, then the number of edges of $G$ is at most
\[|E|=O_{d}\paren{u^{\frac2{d+1}}D^{\frac{2d}{d+1}}(|P||Q|)^{\frac{d}{d+1}}+u|P|+u|Q|}.\]
\end{theorem}

\begin{proof}
We first deal with the case that $u^{d-1}|Q|> D^{d-1}|P|^d$. 
In this case, \cref{thm:zarankiewicz-weak} gives the bound
\[|E|\lesssim_d u^{1/d}D |P||Q|^{1-1/d}+u|Q|.\]
We will upper bound the first term by $u^{\frac2{d+1}}D(|P||Q|)^{\frac{d}{d+1}}$, which suffices as $1\leq \frac{2d}{d+1}$.
To do so, simply note that
\[u^{\frac1d}D |P||Q|^{1-\frac1d}=u^{\frac1d}D\left(\frac{\abs{P}^d}{\abs{Q}}\right)^{\frac{1}{d(d+1)}}(|P||Q|)^{\frac{d}{d+1}}<u^{\frac{2}{d+1}}D(|P||Q|)^{\frac{d}{d+1}}\]
as $\abs{P}^d/\abs{Q}<u^{d-1}/D^{d-1}\leq u^{d-1}$ in this case.

Therefore we may now assume that $u^{d-1}\abs{Q}\leq D^{d-1}\abs{P}^d$.
We apply multilevel polynomial partitioning, \cref{thm:main-partition}, to $P\subset\R^d$ with parameter \[A=u^{-1/(d+1)}D^{1/(d+1)}|P|^{d/(d^2-1)}|Q|^{-1/(d^2-1)}.\] Note that by assumption $A\geq 1$.

For each $q\in Q$, consider the semialgebraic neighborhood set
\[N_q=\{x\in\R^d:\Phi(\sgn(g_1(x,q)),\ldots,\sgn(g_t(x,q)))=1\}.\]
For $p\in P$ we have $(p,q)\in E$ if and only if $p\in N_q$.

For each cell $C\in\cC$ we have a set $P_C$ produced by multilevel polynomial partitioning. Now for each cell, define a partition
\[Q=Q_{C,\mathsf{none}}\sqcup Q_{C,\mathsf{some}}\sqcup Q_{C,\mathsf{all}}\]
where $Q_{C,\mathsf{none}}=\{q\in Q:N_q\cap C=\emptyset\}$ and $Q_{C,\mathsf{all}}=\{q\in Q:N_q\cap C=C\}$ and the remaining vertices lie in $Q_{C,\mathsf{some}}$. Note in particular that $E[P_C\times Q_{C,\mathsf{none}}]$ is empty and $E[P_C\times Q_{C,\mathsf{all}}]$ is complete. Thus we can write
\[|E|=\sum_{C\in\cC}|E[P_C\times Q_{C,\mathsf{some}}]|+|E[P_C\times Q_{C,\mathsf{all}}]|.\]

We bound the first sum using \cref{thm:zarankiewicz-weak}. For $V\in\cV$, define $i_V(A)$ so that $V\in\cV_{i_V(A)}$. Now
\begin{align*}
\sum_{C\in\cC}|E&[P_C\times Q_{C,\mathsf{some}}]|\\
&=\sum_{V\in\cV}\sum_{C\in\cC_V}|E[P_C\times Q_{C,\mathsf{some}}]|\\
&\lesssim_{d}\sum_{V\in\cV}\sum_{C\in\cC_V}u^{1/d}D|P_C||Q_{C,\mathsf{some}}|^{1-1/d}+u|Q_{C,\mathsf{some}}|\\
&\lesssim_{d}\sum_{V\in\cV}\left(u^{1/d}D\paren{\frac{|P_V|}{A^{i_V(A)}\prod_{j=1}^{d-i_V(A)}\deg Q_j^{(V)}}}^{1-1/d}\sum_{C\in\cC_V}|P_C|^{1/d}|Q_{C,\mathsf{some}}|^{1-1/d}\right.\\&\qquad\qquad\qquad\qquad\left.+u\sum_{C\in\cC_V}|Q_{C,\mathsf{some}}|\right)\tag*{\tiny [by \cref{eq:size-of-parts}]}\\
&\leq\sum_{V\in\cV}\paren{u^{1/d}D|P_V|\paren{\frac{\sum_{C\in\cC_V}|Q_{C,\mathsf{some}}|}{A^{i_V(A)}\prod_{j=1}^{d-i_V(A)}\deg Q_j^{(V)}}}^{1-1/d}+u\sum_{C\in\cC_V}|Q_{C,\mathsf{some}}|}\tag*{\tiny [H\"older's inequality]}\\
\end{align*}

For $q\in Q$ and $V\in\cV$ we count the number of cells $C\in\cC_V$ such that $q\in Q_{C,\mathsf{some}}$. If $N_q$ is neither entirely contained in nor entirely disjoint from $C$, we see that one of $\sgn(g_1(x,q)),\ldots,\sgn(g_t(x,q))$ is not constant on $C$. Thus by \cref{thm:number-of-bad-tuples} (applied with $k=1$ and $g(x)=g_r(x,q)$), we see that $q\in Q_{C,\mathsf{some}}$ for at most $\lesssim_{d} DA^{i_V(A)-1}\prod_{j=1}^{d-i_V(A)}\deg Q_j^{(V)}$ choices of $C\in\cC_V$. Summing over $q\in Q$ we conclude that
\[\sum_{C\in\cC_V}|Q_{C,\mathsf{some}}|\lesssim_{d}|Q|DA^{i_V(A)-1}\prod_{j=1}^{d-i_V(A)}\deg Q_j^{(V)}.\]
Combining this with the above inequalities gives
\begin{align*}
\sum_{C\in\cC}|E[P_C\times Q_{C,\mathsf{some}}]|
&\lesssim_{d}\sum_{V\in\cV}\left(u^{1/d}D|P_V|\paren{\frac{|Q|DA^{i_V(A)-1}\prod_{j=1}^{d-i_V(A)}\deg Q_j^{(V)}}{A^{i_V(A)}\prod_{j=1}^{d-i_V(A)}\deg Q_j^{(V)}}}^{1-1/d}\right.\\&\qquad\qquad\qquad\qquad\left.+u|Q|DA^{i_V(A)-1}\prod_{j=1}^{d-i_V(A)}\deg Q_j^{(V)}\right)\\
&\lesssim_{d}\sum_{V\in\cV}\left(u^{1/d}D^{2-1/d}|P_V|\paren{\frac{|Q|}{A}}^{1-1/d}+u|Q|DA^{i_V(A)-1}\prod_{j=1}^{d-i_V(A)}\deg Q_j^{(V)}\right)\\
&\lesssim_{d} u^{1/d}D^{2-1/d}|P|\paren{\frac{|Q|}{A}}^{1-1/d}+u|Q|DA^{d-1}.\tag*{\tiny [by \cref{eq:total-size-of-point-sets} and \cref{eq:total-deg}]}
\end{align*}
By our choice of $A$, both terms are $u^{2/(d+1)}D^{2d/(d+1)}|P|^{d/(d+1)}|Q|^{d/(d+1)}$.

Now we bound the second sum. Note that since $G$ is $K_{u,u}$-free, we have either $|P_C|<u$ or $|Q_{C,\mathsf{all}}|<u$ for each $C\in\cC$. We then have the easy bound
\begin{align*}
\sum_{C\in\cC:|P_C|\geq u}|E[P_C\times Q_{C,\mathsf{all}}]|
&<\sum_{C\in\cC:|P_C|\geq u}u|P_C|\lesssim_d u|P|.
\end{align*}
To finish, we define
\[P'=\bigcup_{C\in\cC:|P_C|<u}P_C\]
so that
\[\sum_{C\in\cC:|P_C|<u}|E[P_C\times Q_{C,\mathsf{all}}]|\leq |E[P'\times Q]|\lesssim_d u^{1/d}D|P'|^{1-1/d}|Q|+u|P'|\]
by \cref{thm:zarankiewicz-weak}.

We have the bound $|P'|<u|\cC|\lesssim_d uA^d$ by \cref{thm:number-of-cells} together with \cref{eq:total-deg}. Using this bound and the trivial bound $|P'|\leq |P|$, we see that the second sum is bounded by
\[O_d\paren{u^{1/d}D(uA^d)^{1-1/d}|Q|+u|P|}.\]
As the first term simplifies to $u\abs{Q}DA^{d-1}$, which has already appeared, the proof is complete.
\end{proof}

Now we give the general proof. The proof is by induction on the uniformity $k$. Given the result for $(k-1)$-uniform hypergraphs, we bootstrap it to a weak Zarankiewicz bound for $k$-uniform hypergraphs. This upgrading procedure uses our Tur\'an result, specifically \cref{thm:generalized-turan} for $\ell=1$. Then we upgrade the weak bound to the full theorem via multilevel polynomial partitioning. The following is the full Zarankiewicz result.

\begin{theorem}
\label{thm:zarankiewicz-full}
Let $H=(P_1\sqcup \cdots\sqcup P_k,E)$ be a semialgebraic hypergraph in $\R^d$ with total degree $D$. If $H$ is $K^{(k)}_{u,\ldots,u}$-free, then the number of edges of $H$ is at most
\[|E|=O_{d,k}\paren{\abs{P_1}\cdots \abs{P_k}\sum_{s=1}^{k}u^{\frac{s+(\beta_s+(k-s)\alpha_s)d}{(s-1)d+1}}D^{\frac{\alpha_sd}{(s-1)d+1}}\sum_{I\in \binom{[k]}{s}}\left(\prod_{i\in I}\abs{P_i}\right)^{-\frac{1}{(s-1)d+1}}}\]
where
\[\alpha_s = \frac{(s-1)(s+2)}{2}\qquad\text{and}\qquad \beta_s=\frac{(s-2)(s-1)(s+3)}{6}.\]
\end{theorem}

\begin{proof}
We will prove the statement by induction on $k$. When $k=1$, the result is trivial since a $K_u^{(1)}$-free 1-uniform hypergraph is exactly one satisfying $|E|<u$. Now assume that $k\geq 2$ and that the statement holds for $k-1$.

We will first prove a weak Zarankiewicz result from the inductive hypothesis. Suppose that there are $\varepsilon \abs{P_1}\cdots \abs{P_k}$ edges in $E$. By \cref{thm:generalized-turan} applied with $\ell=1$, there exists a set $S_k\subseteq P_k$ with size at least $\gtrsim_{d,k}(\varepsilon/D)^d\abs{P_k}$ and a set $T\subseteq P_1\times\cdots\times P_{k-1}$ with size at least $(\varepsilon/3) \abs{P_1}\cdots\abs{P_{k-1}}$ such that $T\times S_k\subseteq E$. 

If $\abs{S_k}<u$, then $u>|S_k|\gtrsim_{d,k}(\epsilon/D)^d|P_k|$, implying $\varepsilon \lesssim_{d,k} u^{1/d}D\abs{P_k}^{-1/d}$.
Now if $\abs{S_k}\geq u$, choose $u$ distinct elements $q_1,\ldots,q_u\in S_k$ arbitrarily and define the following $(k-1)$-uniform semialgebraic hypergraph $(P_1\sqcup\cdots\sqcup P_{k-1}, T')$. For each $q\in P_k$ define the semialgebraic neighborhood $N_q\subseteq (\R^d)^{k-1}$ to be
\[N_q=\{x\in (\R^d)^{k-1}:\Phi(\sgn(g_1(x,q)),\ldots, \sgn(g_t(x,q)))=1\}.\]
Then let $T'$ be the set of $(k-1)$-tuples in $P_1\times\cdots\times P_{k-1}$ that lie in $N_{q_1}\cap \cdots\cap N_{q_u}$.
It is clear that $T'$ is the edge set of a $(k-1)$-uniform semialgebraic hypergraph in $\R^d$ with total degree at most $uD$.
We also know that $T\subseteq T'$.
Finally, since $T'$ is the common neighborhood of $q_1,\ldots,q_u$ we know that $T'$ is $K^{(k-1)}_{u,\ldots,u}$-free.
Therefore by the inductive hypothesis,
\[|T|\leq \abs{T'}\lesssim_{d,k}\abs{P_1}\cdots \abs{P_{k-1}}\sum_{s=1}^{k-1}u^{\frac{s+(\beta_s+(k-1-s)\alpha_s)d}{(s-1)d+1}}(uD)^{\frac{\alpha_sd}{(s-1)d+1}}\sum_{I\in \binom{[k-1]}{s}}\left(\prod_{i\in I}\abs{P_i}\right)^{-\frac{1}{(s-1)d+1}}.\]
As a consequence,
\[\varepsilon\lesssim_{d,k}\sum_{s=1}^{k-1}u^{\frac{s+(\beta_s+(k-s)\alpha_s)d}{(s-1)d+1}}D^{\frac{\alpha_sd}{(s-1)d+1}}\sum_{I\in \binom{[k-1]}{s}}\left(\prod_{i\in I}\abs{P_i}\right)^{-\frac{1}{(s-1)d+1}}.\]
Combining the cases where $\abs{S_k}\geq u$ and where $\abs{S_k}<u$, we deduce that
\begin{equation}
\label{eq:weak-zarankiewicz-full}
\begin{split}
|E|\lesssim_{d,k}u^{1/d}D&\abs{P_1}\cdots \abs{P_{k-1}}\abs{P_k}^{1-1/d}\\&+\abs{P_1}\cdots \abs{P_k}\sum_{s=1}^{k-1}u^{\frac{s+(\beta_s+(k-s)\alpha_s)d}{(s-1)d+1}}D^{\frac{\alpha_sd}{(s-1)d+1}}\sum_{I\in \binom{[k-1]}{s}}\left(\prod_{i\in I}\abs{P_i}\right)^{-\frac{1}{(s-1)d+1}}.
\end{split}
\end{equation}
Note that this bound is too weak to deduce the desired result because the first term is potentially too large. However, the second term is smaller than the desired bound.

Now we upgrade this weak Zarankiewicz result to prove the desired bound for uniformity $k$. Apply multilevel polynomial partitioning, \cref{thm:main-partition}, to $P_1,\ldots,P_{k-1}$ with some parameter $A\geq 1$ that will be specified later. This produces data $\cV^{(1)},\ldots,\cV^{(k-1)}$.
 
For $V_1\in\cV^{(1)},\ldots, V_{k-1}\in\cV^{(k-1)}$ and each tuple of cells $C\in\cC_{V_1}\times\cdots\times\cC_{V_{k-1}}$, define the partition
\[P_k=P_{C,\mathsf{none}}\sqcup P_{C,\mathsf{some}}\sqcup P_{C,\mathsf{all}}\]
as before: $P_{C,\mathsf{none}}$ consists of points $q\in P_k$ with $N_q\cap C=\emptyset$ and $P_{C,\mathsf{all}}$ consists of points $q\in P_k$ with $N_q\cap C = C$ and $P_{C,\mathsf{some}}$ consists of the remaining points in $P_k$.
We get
\[|E|=\sum_C|E[(P_1)_{C_1}\times\cdots\times (P_{k-1})_{C_{k-1}}\times P_{C,\mathsf{some}}]|+|E[(P_1)_{C_1}\times\cdots\times (P_{k-1})_{C_{k-1}}\times P_{C,\mathsf{all}}]|.\]

To deal with the first term we apply the weak Zarankiewicz result. This gives the upper bound
\[|E[(P_1)_{C_1}\times\cdots\times (P_{k-1})_{C_{k-1}}\times P_{C,\mathsf{some}}]|\lesssim_{d,k} E_0(C)+\sum_{\emptyset\neq I\subseteq [k-1]}E_I(C)\]
where we define
\[E_0(C)=u^{1/d}D\abs{(P_1)_{C_1}}\cdots \abs{(P_{k-1})_{C_{k-1}}}\abs{P_{C,\mathsf{some}}}^{1-1/d}\]
and
\[E_I(C)=\abs{(P_1)_{C_1}}\cdots \abs{(P_{k-1})_{C_{k-1}}}\abs{P_{C,\mathsf{some}}}u^{\frac{s+(\beta_s+(k-s)\alpha_s)d}{(s-1)d+1}}D^{\frac{\alpha_sd}{(s-1)d+1}}\left(\prod_{i\in I}\abs{(P_i)_{C_i}}\right)^{-\frac{1}{(s-1)d+1}}\]
for each tuple of cells $C\in\cC_{V_1}\times\cdots\times\cC_{V_{k-1}}$ and each $I\subseteq[k-1]$ with $|I|=s\neq 0$.

First fix a choice $V=(V_1,\ldots, V_{k-1})\in\cV^{(1)}_{i_1}\times\cdots\times\cV^{(k-1)}_{i_{k-1}}$. We will bound the sum over $C\in \cC_{V_1}\times\cdots\times \cC_{V_{k-1}}$. By \cref{eq:size-of-parts}, we know that 
\[\abs{(P_\ell)_{C_\ell}}\lesssim_{d}\frac{\abs{(P_\ell)_{V_\ell}}}{A^{i_{\ell}}\prod_{j=1}^{d-i_{\ell}}\deg Q_j^{(V_\ell)}}\]
for each $1\leq \ell\leq k-1$. For each $q\in P_k$, we count the number of cells $C$ so that $q\in P_{C,\mathsf{some}}$. This only occurs if there exists some $1\leq r\leq t$ so that $\sgn(g_r(x_{11},\ldots,x_{k-1,d},q_1,\ldots,q_d))$ is not constant on $C\subseteq \R^{(k-1)d}$. Applying \cref{thm:number-of-bad-tuples} to bound the number of such cells and summing over $r$ and $q\in P_k$, we conclude that
\begin{equation}
\label{eq:sum-of-p-c-some}
\sum_{C\in\cC_{V_1}\times\cdots\times\cC_{V_{k-1}}}|P_{C,\mathsf{some}}|\lesssim_{d,k}\abs{P_k}DA^{i_1+\cdots+i_{k-1}-1}\prod_{\ell=1}^{k-1}\prod_{j=1}^{d-i_{\ell}(A)}\deg Q_j^{(V_\ell)}
\end{equation}
for each $V\in\cV^{(1)}_{i_1}\times\cdots\times\cV^{(k-1)}_{i_{k-1}}$.

Thus we get the bound
\begin{align*}
\sum_{C\in \cC_{V_1}\times\cdots\times \cC_{V_{k-1}}}E_0(C) &\lesssim_{d,k}u^{1/d}D\left(\prod_{\ell=1}^{k-1}\frac{\abs{(P_\ell)_{V_\ell}}}{A^{i_{\ell}}\prod_{j=1}^{d-i_{\ell}}\deg Q_j^{(V_\ell)}}\right)^{1-1/d}\\&\qquad\qquad\qquad\cdot\sum_{C\in \cC_{V_1}\times\cdots\times \cC_{V_{k-1}}}\left(\prod_{\ell=1}^{k-1}\abs{(P_{\ell})_{C_{\ell}}}\right)^{1/d}\abs{P_{C,\mathsf{some}}}^{1-1/d}\\
&\leq u^{1/d}D\left(\prod_{\ell=1}^{k-1}\frac{\abs{(P_\ell)_{V_\ell}}}{A^{i_{\ell}}\prod_{j=1}^{d-i_{\ell}}\deg Q_j^{(V_\ell)}}\right)^{1-1/d}\paren{\sum_{C\in \cC_{V_1}\times\cdots\times \cC_{V_{k-1}}}\prod_{\ell=1}^{k-1}\abs{(P_{\ell})_{C_{\ell}}}}^{1/d}\\&\qquad\qquad\qquad\cdot\paren{\sum_{C\in \cC_{V_1}\times\cdots\times \cC_{V_{k-1}}}\abs{P_{C,\mathsf{some}}}}^{1-1/d}\tag*{\tiny [H\"older's inequality]}\\
&\lesssim_{d,k}u^{1/d}D\left(\prod_{\ell=1}^{k-1}\frac{\abs{(P_\ell)_{V_\ell}}}{A^{i_{\ell}}\prod_{j=1}^{d-i_{\ell}}\deg Q_j^{(V_\ell)}}\right)^{1-1/d}\paren{\prod_{\ell=1}^{k-1}\abs{(P_{\ell})_{V_{\ell}}}}^{1/d}\\&\qquad\qquad\qquad\cdot\paren{\abs{P_k}DA^{i_1+\cdots+i_{k-1}-1}\prod_{\ell=1}^{k-1}\prod_{j=1}^{d-i_{\ell}(A)}\deg Q_j^{(V_\ell)}}^{1-1/d}\tag*{\tiny [by \cref{eq:sum-of-p-c-some}]}\\
&=u^{1/d}D^{2-1/d}\abs{(P_1)_{V_1}}\cdots \abs{(P_{k-1})_{V_{k-1}}}\paren{\frac{\abs{P_k}}A}^{1-1/d}.
\end{align*}

Then summing over $V\in\cV^{(1)}\times\cdots\times\cV^{(k-1)}$ and applying \cref{eq:total-size-of-point-sets}, we conclude that
\[\sum_{V\in\cV^{(1)}\times\cdots\times\cV^{(k-1)}}\sum_{C\in \cC_{V_1}\times\cdots\times \cC_{V_{k-1}}}E_0(C)\lesssim_{d,k}u^{1/d}D^{2-1/d}\abs{P_1}\cdots \abs{P_{k-1}}\paren{\frac{\abs{P_k}}A}^{1-1/d}.\]
Let $F(A)$ denote the expression on the right-hand side above.

Now we turn to bounding $E_I(C)$. Again fix a choice $V=(V_1,\ldots, V_{k-1})\in\cV^{(1)}_{i_1}\times\cdots\times\cV^{(k-1)}_{i_{k-1}}$. We will bound the sum over $C\in \cC_{V_1}\times\cdots\times \cC_{V_{k-1}}$. Applying the bound on $|(P_\ell)_{V_\ell}|$ we see that
\[\begin{split}
E_I(C)\lesssim_{d,k}u^{\frac{s+(\beta_s+(k-s)\alpha_s)d}{(s-1)d+1}}D^{\frac{\alpha_sd}{(s-1)d+1}}\abs{P_{C, \mathsf{some}}}&\left(\prod_{\ell\in I}\frac{\abs{(P_\ell)_{V_\ell}}}{A^{i_{\ell}}\prod_{j=1}^{d-i_{\ell}}\deg Q_j^{(V_\ell)}}\right)^{1-\frac{1}{(s-1)d+1}}\\
&\cdot\prod_{\ell\in[k-1]\setminus I}\frac{\abs{(P_\ell)_{V_\ell}}}{A^{i_{\ell}}\prod_{j=1}^{d-i_{\ell}}\deg Q_j^{(V_\ell)}}.
\end{split}\]
Together with \cref{eq:sum-of-p-c-some} and the fact that $\alpha_s+(s-1) = \alpha_{s+1}-2$, we get
\[\begin{split}
\sum_{C\in \cC_{V_1}\times\cdots\times \cC_{V_{k-1}}}E_I(C)\lesssim_{d,k}\abs{P_k} & u^{\frac{s+(\beta_s+(k-s)\alpha_s)d}{(s-1)d+1}}D^{\frac{1+(\alpha_{s+1}-2)d}{(s-1)d+1}}A^{-1}\\
&\cdot\prod_{\ell\in I}|(P_\ell)_{V_\ell}|\left(\frac{A^{i_{\ell}}\prod_{j=1}^{d-i_{\ell}}\deg Q_j^{(V_\ell)}}{\abs{(P_{\ell})_{V_{\ell}}}}\right)^{\frac{1}{(s-1)d+1}}\prod_{\ell\in[k-1]\setminus I}|(P_\ell)_{V_\ell}|.
\end{split}\]
For each $\ell\in I$, we know that
\begin{align*}
\sum_{i_\ell=0}^d\sum_{V_{\ell}\in \cV^{(\ell)}_{i_\ell}}|&(P_{\ell})_{V_{\ell}}|\left(\frac{A^{i_{\ell}}\prod_{j=1}^{d-i_{\ell}}\deg Q_j^{(V_\ell)}}{\abs{(P_{\ell})_{V_{\ell}}}}\right)^{\frac{1}{(s-1)d+1}}\\
&=\sum_{i_\ell=0}^d\sum_{V_{\ell}\in \cV^{(\ell)}_{i_\ell}}\abs{(P_{\ell})_{V_{\ell}}}^{1-\frac{1}{(s-1)d+1}}\left(A^{i_{\ell}}\prod_{j=1}^{d-i_{\ell}}\deg Q_j^{(V_\ell)}\right)^{\frac{1}{(s-1)d+1}}\\
&\leq \sum_{i_\ell=0}^d\paren{\sum_{V_{\ell}\in \cV^{(\ell)}_{i_\ell}}\abs{(P_{\ell})_{V_{\ell}}}}^{1-\frac{1}{(s-1)d+1}}\left(\sum_{V_{\ell}\in \cV^{(\ell)}_{i_\ell}}A^{i_{\ell}}\prod_{j=1}^{d-i_{\ell}}\deg Q_j^{(V_\ell)}\right)^{\frac{1}{(s-1)d+1}}\tag*{\tiny [H\"older's inequality]}\\
&\lesssim_d \sum_{i_\ell=0}^d\abs{P_{\ell}}^{1-\frac{1}{(s-1)d+1}}\left(A^d\right)^{\frac{1}{(s-1)d+1}}\tag*{\tiny [by \cref{eq:total-size-of-point-sets} and \cref{eq:total-deg}]}\\
&\lesssim_d\abs{P_{\ell}}\left(\frac{A^d}{\abs{P_{\ell}}}\right)^{\frac{1}{(s-1)d+1}}.
\end{align*}

Applying the above for $\ell\in I$ and applying \cref{eq:total-size-of-point-sets} for $\ell\not\in I$, we conclude that
\begin{align*}
\sum_{V\in\cV^{(1)}\times\cdots\times\cV^{(k-1)}}&\sum_{C\in \cC_{V_1}\times\cdots\times \cC_{V_{k-1}}}E_I(C)\\
&\lesssim_{d,k}|P_k|u^{\frac{s+(\beta_s+(k-s)\alpha_s)d}{(s-1)d+1}}D^{\frac{1+(\alpha_{s+1}-2)d}{(s-1)d+1}}A^{-1}\prod_{\ell\in I}\abs{P_{\ell}}\left(\frac{A^d}{\abs{P_{\ell}}}\right)^{\frac{1}{(s-1)d+1}}\prod_{\ell\in[k-1]\setminus I}|P_\ell|\\
&=\abs{P_1}\cdots \abs{P_k}u^{\frac{s+(\beta_s+(k-s)\alpha_s)d}{(s-1)d+1}}D^{\frac{1+(\alpha_{s+1}-2)d}{(s-1)d+1}}A^{\frac{d-1}{(s-1)d+1}}\prod_{\ell\in I}\abs{P_{\ell}}^{-\frac1{(s-1)d+1}}
\end{align*}
for each $I\subseteq[k-1]$ with $|I|=s\neq 0$. Let $G_I(A)$ denote the right-hand side above.

What we have shown is that for each $A\geq 1$,
\[\sum_{C\in\cC^{(1)}\times\cdots\times\cC^{(k-1)}}|E[(P_1)_{C_1}\times\cdots\times (P_{k-1})_{C_{k-1}}\times P_{C,\mathsf{some}}]|\lesssim_{d,k}F(A)+\sum_{\emptyset\neq I\subseteq [k-1]}G_I(A).\]
It remains to bound $\sum_C|E[(P_1)_{C_1}\times\cdots\times (P_{k-1})_{C_{k-1}}\times P_{C,\mathsf{all}}]|$.
Note that for each $C$, either there exists $1\leq \ell\leq k-1$ with $\abs{(P_{\ell})_{C_{\ell}}}<u$ or $\abs{P_{C,\mathsf{all}}}<u$. Define
\[P'_\ell=\bigcup_{C_\ell\in\cC^{(\ell)}:|(P_\ell)_{C_\ell}|<u}(P_\ell)_{C_\ell}\]
for each $1\leq \ell\leq k-1$. It is then clear that
\[\begin{split}
\sum_{C\in\cC^{(1)}\times\cdots\times\cC^{(k-1)}}|E[(P_1)_{C_1}\times\cdots\times (P_{k-1})_{C_{k-1}}\times &P_{C,\mathsf{all}}]|<u|P_1|\cdots|P_{k-1}|\\
&+\sum_{\ell=1}^{k-1}|E[P_1\times\cdots\times P'_\ell\times \cdots\times P_{k-1}\times P_k]|.
\end{split}\]
Note that the first term is one that appears on the right-hand side of the desired bound. (It is the $I=\{k\}$ term.) We now bound the second term.

As $|\cC^{(\ell)}|\lesssim_d A^d$ by \cref{thm:number-of-cells,eq:total-deg}, we see that $\abs{P'_{\ell}}\lesssim_d uA^d$ for each $1\leq \ell\leq k-1$. We also have the trivial bound $|P'_\ell|\leq |P_\ell|$. We now apply the weak Zarankiewicz bound, \cref{eq:weak-zarankiewicz-full}, with $P_{\ell}'$ as the last part. This gives
\begin{align*}
|E[P_1\times\cdots\times P'_\ell\times \cdots&\times P_{k-1}\times P_k]|
\lesssim_{d,k} u^{1/d}D(uA^d)^{1-1/d}\prod_{i\in[k]\setminus\{\ell\}}\abs{P_i}\\
&+\prod_{i\in[k]}\abs{P_{i}}\cdot \sum_{s=1}^{k-1}u^{\frac{s+(\beta_s+(k-s)\alpha_s)d}{(s-1)d+1}}D^{\frac{\alpha_sd}{(s-1)d+1}}\sum_{I\in \binom{[k]\setminus\{\ell\}}{s}}\left(\prod_{i\in I}\abs{P_i}\right)^{-\frac{1}{(s-1)d+1}}.
\end{align*}
Note that the second term consists only of terms that appear on the right-hand side of the desired bound. For the first term, note that it simplifies to $uDA^{d-1}\prod_{i\neq \ell}\abs{P_i}$ which is equal to $G_{\{\ell\}}(A)$. Putting everything together, we have shown that if there exists $A\geq 1$ so that
\begin{equation*}
F(A)+\sum_{\emptyset \neq I\subseteq [k-1]}G_I(A)
\lesssim_{d,k}|P_1|\cdots|P_k|\sum_{s=1}^{k}u^{\frac{s+(\beta_s+(k-s)\alpha_s)d}{(s-1)d+1}}D^{\frac{\alpha_sd}{(s-1)d+1}}\sum_{I\in \binom{[k]}{s}}\left(\prod_{i\in I}\abs{P_i}\right)^{-\frac{1}{(s-1)d+1}}\label{eq:Zarankiewicz-claim},
\end{equation*}
then we can close the induction.

Observe that as a function of $A$, the map $F\colon\R_{>0}\to\R_{>0}$ is decreasing and surjective while $G_I\colon\R_{>0}\to\R_{>0}$ is increasing and surjective for all $\emptyset \neq I\subseteq [k-1]$. Thus for every $\emptyset \neq I\subseteq [k-1]$ there exists a unique $A_I>0$ so that $F(A_I)=G_I(A_I)$. Set $\abs{I}=s\in[k-1]$. From the fact that $\beta_s+(k-s)\alpha_s=\beta_{s+1}+(k-(s+1))\alpha_s\leq \beta_{s+1}+(k-(s+1))\alpha_{s+1}$, we can compute
\begin{align*}
    F(A_I) &= F(A_I)^{\frac{d}{sd+1}}G_I(A_I)^{\frac{(s-1)d+1}{sd+1}}\\
    &\leq  |P_1|\cdots|P_k|u^{\frac{(s+1)+(\beta_{s+1}+(k-(s+1))\alpha_{s+1})d}{sd+1}}D^{\frac{\alpha_{s+1}d}{sd+1}}\left(\prod_{\ell\in I\cup\{k\}}\abs{P_\ell}\right)^{-\frac{1}{sd+1}},
\end{align*}
which is the term corresponding to $I\cup\{k\}$ on the right-hand side of the desired bound.
Set $A=\min_{\emptyset \neq I\subseteq [k-1]}A_I$. If $A\geq 1$, then we have $G_I(A_I)=F(A_I)\leq F(A)$ for all $\emptyset \neq I\subseteq [k-1]$ and we also have that $F(A)$ is equal to the $I\cup\{k\}$ term on the right-hand side of the desired bound where $I$ minimizes $A_{I}$.

Therefore, it remains to deal with the case where $A<1$. In this case, we can find $\emptyset\neq I\subseteq [k-1]$ with $A_I<1$. Set $|I|=s\neq 0$. We know that $F(1)<F(A_I)=G_I(A_I)<G_I(1)$ and so
\begin{align*}
u^{\frac1d}D|P_1|\cdots|P_{k-1}||P_k|^{1-\frac1d}
&\leq F(1)
<F(1)^{\frac{d}{sd+1}}G_I(1)^{\frac{(s-1)d+1}{sd+1}}\\
&\leq |P_1|\cdots|P_k|u^{\frac{(s+1)+(\beta_{s+1}+(k-(s+1))\alpha_{s+1})d}{sd+1}}D^{\frac{\alpha_{s+1}d}{sd+1}}\left(\prod_{\ell\in I\cup\{k\}}\abs{P_\ell}\right)^{-\frac{1}{sd+1}}.
\end{align*}
In this case, we apply the weak Zarankiewicz bound, \cref{eq:weak-zarankiewicz-full}, to the whole hypergraph $H$. The above inequality show that the first term in that bound is less than one of the terms in the right-hand side of the desired bound while we already noted that all the other terms in that bound are sufficiently small to deduced the desired inequality. This completes the induction.
\end{proof}

Finally, the symmetric version stated in the introduction follows immediately from this result.

\begin{proof}[Proof of \cref{thm:zarankiewicz}]
The desired bound for general $k$ follows from \cref{thm:zarankiewicz-full} by setting $|P_1|=\cdots=|P_k|=N$. Now specializing to $k=2$ gives
\[\abs{E}\lesssim_d N^2\paren{\paren{\frac uN}+D^{\frac{2d}{d+1}}\paren{\frac uN}^{\frac2{d+1}}}.\]
Note that it suffices to prove the claimed bound for $u\leq N$ since otherwise it is weaker than the trivial bound $|E|\leq N^2$. Then since $D\geq 1$ and $u/N\leq 1$, the second term above dominates, giving the claimed bound.
\end{proof}

\section{Additional applications}
\label{sec:applications}

\subsection{The Erd\H{o}s unit distance problem}
\label{ssec:unit-distance}
The Erd\H{o}s unit distance problem asks for the maximum number of pairs of points at distance 1 among a set of $N$ points in $\R^d$. For $d=2,3$, it is conjectured that an appropriately scaled section of the integer lattice asymptotically maximizes the number of unit distances. For $d\geq 4$ a set of $N$ points can span $\Theta(N^2)$ unit distances. This can be achieved by placing the points on two circles lying in orthogonal planes. One can make this problem non-trivial by ruling out this configuration. This was studied by Oberlin and Oberlin who proved a bound of $O_d(N^{2-1/d})$ if no $d$ of the points are contained in a $(d-2)$-dimensional affine subspace \cite{OO15}.

Under the weaker restriction that the unit distance graph is $K_{u,u}$-free, a better bound follows from the semialgebraic Zarankiewicz machinery. From the work of Fox--Pach--Sheffer--Suk--Zahl one deduces the bound of $O_{d,u}(N^{2d/(d+1)+o(1)})$ unit distances \cite{FPSSZ17}. Do \cite{Do19} improved this to $O_{d}(uN^{2d/(d+1)+o(1)})$. Frankl and Kupavskii \cite{FK21}, using cutting techniques that are tailored to the unit distance graph, improved the bound to $O_{d}(u^{2/(d+1)}N^{2d/(d+1)+o(1)})$. As a corollary of our semialgebraic Zarankiewicz result, \cref{thm:zarankiewicz}, we remove the $o(1)$-loss in the exponent. Furthermore we give a short geometric argument that any $K_{u,u}$ in the unit distance graph must come from a pair of orthogonal spheres.

\begin{proposition}
\label{thm:unit-distance}
Let $P\subset\R^d$ be a set of $N$ points so that every sphere of dimension $\floor{d/2}-1$ contains fewer than $u$ points of $P$. Then the number of unit distances spanned by $P$ is at most
\[O_d\paren{u^{\frac{2}{d+1}}N^{\frac{2d}{d+1}}}.\]
\end{proposition}

\begin{proof}
Define $G$ to be the bipartite graph where each vertex set is a copy of $P$ and there is an edge between $p$ and $q$ if they are at distance 1. Clearly $G$ is a semialgebraic graph of total degree 2. Suppose for contradiction that $p_1,\ldots,p_u,q_1,\ldots,q_u$ form a copy of $K_{u,u}$ in $G$.

Let $V$ be the affine subspace of $\R^d$ spanned by $p_1,\ldots, p_u$. Suppose that $\dim V=r$ and without loss of generality, suppose that $V$ is the affine span of $p_1,\ldots,p_{r+1}$.
For each $i\in[r]$, define the hyperplane $H_i\subseteq V$ to be the perpendicular bisector of the segment $\overline{p_ip_{r+1}}$. 
As $p_1,\ldots, p_{r+1}$ affinely span the $r$-dimensional affine subspace $V$, we see that the normal vectors to $H_1,\ldots, H_r$ are linearly independent, showing that they intersect at a unique point $v$. It is then clear that $v$ is equidistant from $p_1,\ldots,p_{r+1}$.

Without loss of generality, assume that $v$ is the origin. Let $W$ be the orthogonal complement of $V$ (i.e., the subspace of dimension $d-r$ which is orthogonal to $V$ and passes through the origin). We claim that if $q$ is equidistant from $p_1,\ldots,p_{r+1}$, then it must lie in $W$. Write $q=q_V+q_W$ where $q_V\in V$ and $q_W\in W$. Then
\[\|p_i-q\|^2=\|p_i-q_V\|^2+\|q_W\|^2.\]
In particular, $q_V\in V$ is equidistant from $p_1,\ldots,p_{r+1}$, implying that $q_V\in H_1\cap \cdots\cap H_r=\{0\}$. Thus $q=q_W\in W$, as desired.

Let $\alpha=\|p_1\|=\cdots=\|p_{r+1}\|$. For $j\in[u]$, we know that $q_j\in W$ so in particular,
\[1=\|p_1-q_j\|^2=\|p_1\|^2+\|q_j\|^2=\alpha^2+\|q_j\|^2.\]
Thus $q_1,\ldots, q_u$ all lie on the $(d-r-1)$-dimensional sphere of radius $\beta$ in $W$ where $\alpha^2+\beta^2=1$.

For $i\in[u]$, we know that $p_i\in V$ by assumption. Then
\[1=\|p_i-q_1\|^2=\|p_i\|^2+\|q_1\|^2=\|p_i\|^2+\beta^2.\]
Thus $p_1,\ldots,p_u$ all lie on the $(r-1)$-dimensional sphere of radius $\alpha$ in $V$. Clearly one of these two spheres has dimension at most $\floor{d/2}-1$, contradicting our assumption on $P$.
Now the desired bound follows from \cref{thm:zarankiewicz}.
\end{proof}

When $d=2,3$, \cref{thm:unit-distance} applies to any set of $N$ points with $u=3$. For $d=2$ this recovers $O(N^{4/3})$, the best-known upper bound \cite{SST84}. For $d=3$ this gives the bound of $O(N^{3/2})$ which was the best-known bound \cite{CEGSW90,KMSS12,Zah13} until recent work of Zahl \cite{Zah19} that used curve cutting to improve the bound to $N^{3/2-1/394+o(1)}$.

\begin{example}
Let $P$ be a $N^{1/d}\times\cdots\times N^{1/d}$ section of the integer lattice, appropriately scaled. By the pigeonhole principle, there exists a scaling so that $P$ spans $\gtrsim_d N^{2-2/d}$ unit distances. By Schwartz--Zippel for varieties (see, e.g, \cite[Lemma 14]{BT12}), every sphere of dimension $\floor{d/2}-1$ contains at most $\lesssim_d N^{(\floor{d/2}-1)/d}$ points of $P$. Thus \cref{thm:unit-distance} gives an upper bound of
\[\lesssim_d  \paren{N^{\paren{\floor{\frac d2}-1}/d}}^{\frac{2}{d+1}}N^{\frac{2d}{d+1}}=\begin{cases}
N^{2-\frac{d+2}{d(d+1)}}\quad&\text{if }d\text{ is even},\\
N^{2-\frac{d+3}{d(d+1)}}&\text{if }d\text{ is odd}.
\end{cases}\]
This is fairly far from the lower bound even for large $d$, i.e., the upper bound approaches $N^{2-1/d}$ for large $d$ while the lower bound is $N^{2-2/d}$.
\end{example}

We can apply the same argument to bound the number of equilateral triangles contained in a set of $N$ points as long as the points do not cluster on a sphere of dimension $\floor{d/3}-1$. Note that this proof bounds not just the number of unit equilateral triangles, but the total number of equilateral triangles of all side lengths. This argument also easily adapts to bound the number of similar copies of any fixed triangle or more generally any fixed simplex of any dimension.

\begin{proposition}
\label{thm:equilateral-triangle}
Let $P\subset\R^d$ be a set of $N$ points so that every sphere of dimension $\floor{d/3}-1$ contains fewer than $u$ points of $P$. Then the number of equilateral triangles spanned by $P$ is at most
\[O_d\paren{u^2N^{\frac{3d}{d+1}}+u^{\frac{2d+3}{2d+1}}N^{\frac{6d}{2d+1}}}.\]
\end{proposition}

\begin{proof}
Let $H$ be the 3-uniform 3-partite hypergraph whose vertex sets are three copies of $P$ with an edge between a triple of vertices if they form an equilateral triangle. This is a semialgebraic hypergraph of total degree 4. It is defined by the equation
\[\begin{split}((x_1,x_2),(y_1,y_2),(z_1,z_2))\in E(H)&\iff (x_1-y_1)^2+(x_2-y_2)^2-(x_1-z_1)^2-(x_2-z_2)^2=0\\&\text{ and }(x_1-y_1)^2+(x_2-y_2)^2-(y_1-z_1)^2-(y_2-z_2)^2=0.\end{split}\]

Suppose for contradiction that $p_1,\ldots,p_u,q_1,\ldots,q_u,r_1,\ldots,r_u\in P$ form a copy of $K_{u,u,u}^{(3)}$ in $H$. By essentially the same argument as in \cref{thm:unit-distance}, this implies that there exist mutually orthogonal affine subspaces $U,V,W$ all passing through the same point $v$ and $\alpha,\beta,\gamma\geq 0$ such that $\alpha^2+\beta^2=\alpha^2+\gamma^2=\beta^2+\gamma^2$ so that $p_1,\ldots, p_u$ lie on a sphere in $U$ of radius $\alpha$, so that $q_1,\ldots, q_u$ lie on a sphere in $V$ of radius $\beta$, and so that $r_1,\ldots,r_u$ lie on a sphere in $W$ of radius $\gamma$. One of $U,V,W$ has dimension at most $\floor{d/3}$, meaning that $u$ of these points lie on the same sphere of dimension at most $\floor{d/3}-1$, contradiction.
Therefore the desired bound follows from \cref{thm:zarankiewicz} and the fact that $uN^2\leq u^{\frac{2d+3}{2d+1}}N^{\frac{6d}{2d+1}}$ as $u,N\geq 1$.
\end{proof}

\begin{remark}
With more work, one can improve this bound to $O_d(u^{3/(2d+1)}N^{6d/(2d+1)})$. For $q_1,\ldots,q_u\subseteq P$, let $H'=(P\sqcup P,T')$ be their common neighborhood, i.e., $T'\subseteq P^2$ is the set of pairs $(p_1,p_2)$ such that $(p_1,p_2,q_i)\in E(H)$ for all $i\in[u]$. In the proof of \cref{thm:zarankiewicz-full}, we use the fact that $H'$ is a semialgebraic graph of total degree $O(u)$. In this case, one can check that $H'$ actually has total degree $O_d(1)$. Now running through the proof of \cref{thm:zarankiewicz-full} with this improved bound on the total degree of $H'$ gives the claimed bound on the number of equilateral triangles.
\end{remark}

\subsection{The Erd\H{o}s--Hajnal problem}
The next two applications involve semialgebraic graphs and hypergraphs which are not partite. These are defined as follows.

We say that the formula $\Phi(\sgn(g_1(x_{11},\ldots,x_{kd})),\ldots,\sgn(g_t(x_{11},\ldots,x_{kd})))$ is symmetric if for each $(x_{11},\ldots,x_{kd})\in\R^{kd}$ the value of 
\[\Phi(\sgn(g_1(x_{\pi(1)1},\ldots,x_{\pi(k)d})),\ldots,\sgn(g_t(x_{\pi(1)1},\ldots,x_{\pi(k)d})))\]
is independent of $\pi\in\mathfrak S_k$. Then we say that $H=(V,E)$ is a $k$-uniform semialgebraic hypergraph if there exists a symmetric $k$-uniform $k$-partite semialgebraic hypergraph $H'=(V\sqcup\cdots\sqcup V,E')$ such that $\{v_1,\ldots,v_k\}$ is an edge of $H$ if and only if $(v_1,\ldots,v_k)$ is an edge of $H'$.

By a quantitative version of Ramsey's theorem proved by Erd\H{o}s and Szekeres \cite{ESz35}, every $N$-vertex graph has a clique or independent set of size $\Omega(\log N)$. Erd\H{o}s later proved that random graphs almost surely match this bound up to a constant multiplicative factor \cite{Erd47}. However, it is believed that many restricted classes of graphs have significantly larger cliques or independent sets. We say that a family of graphs $\cF$ has the Erd\H{o}s--Hajnal property if there exists some $c=c(\cF)>0$ such that every $N$-vertex graph in $\cF$ contains a clique or independent set of size at least $N^c$. The famous Erd\H{o}s--Hajnal conjecture states that the family of induced $H$-free graphs has the Erd\H{o}s--Hajnal property for every $H$.

In the paper which introduced semialgebraic graphs, Alon, Pach, Pinchasi, Radoi\v{c}i\'c, and Sharir proved that the class of semialgebraic graphs of bounded dimension and total degree has the Erd\H{o}s--Hajnal property \cite{APPRS05}. More recent breakthroughs have shown that numerous other families have the Erd\H{o}s--Hajnal property such as graphs of bounded VC-dimension \cite{NSS23f} and induced $P_5$-free graphs \cite{NSS23g}. From our semialgebraic Tur\'an result we improve the quantitative bounds on the Erd\H{o}s--Hajnal property for semialgebraic graphs, showing that such graphs contain a clique or independent set of size $N^{\Omega(\log\log D/d\log D)}$.
We also show that this exponent is tight up to a constant factor.

\begin{proposition}
\label{thm:erdos-hajnal-lower}
Every $N$-vertex semialgebraic graph in $\R^d$ with total degree $D$ contains a clique or independent set of size at least
\[N^{\paren{\frac12-o(1)}\frac{\log\log D}{d\log D}}\]
where the $o(1)$ term goes to zero for each fixed $d$ as $D$ goes to infinity.
\end{proposition}

\begin{proof}
This follows via the same argument as the original proof of Alon, Pach, Pinchasi, Radoi\v{c}i\'c, and Sharir \cite[Section 3.1]{APPRS05}. We use our improved Tur\'an result, \cref{thm:turan}, to improve the quantitative bounds. 

Recall that the family of cographs is defined to be the minimal family of graphs that is closed under disjoint unions, complements, and contains $K_1$. Fix $d, D$ and define $h(N)$ to be the largest integer $h$ such that any semialgebraic graph in $\R^d$ with total degree $D$ and at least $N$ vertices contains a cograph of size at least $h$. Clearly $h(1)=1$. 

Now let $G$ be such a graph with $N\geq 2$ vertices. Pick a bipartition $V(G)=P\sqcup Q$ into parts of size $\floor{N/2},\ceil{N/2}$. Either $G$ or its complement contains at least $|P||Q|/2$ edges in $P\times Q$. We now apply \cref{thm:turan} to either $G[P\times Q]$ or $\bar{G}[P\times Q]$ with $\epsilon=1/2$. Since both are semialgebraic bipartite graphs in $\R^d$ of total degree $D$, we find $S\subseteq P$ and $T\subseteq Q$ such that $S\times T$ is complete or empty in $G$ and such that $|S|\geq N/(C_dD^d)$ and $|T|\geq N/C$. If $S'\subseteq S$ is a cograph in $G[S]$ and $T'\subseteq T$ is a cograph in $G[T]$, then $S'\sqcup T'$ is a cograph in $G$. Therefore we conclude that
\[h(N)\geq h\paren{\frac N{C}}+h\paren{\frac{N}{C_dD^d}}\]
holds for all $N>1$ for some constants $C,C_d>1$.

Given $N, D, d$, define $t$ to be the greatest positive integer so that 
\[N\paren{\frac1C}^{td\log D/\log\log D}\paren{\frac{1}{C_dD^d}}^{t}\geq 1.\]
Expanding out the previous recurrence, we see that 
\[h(N)\geq \binom{\floor{td\log D/\log\log D}+t}{t}\geq (d\log D/\log\log D)^t.\]
Since
\[t=\floor{\frac{\log N}{\log(C_dD^d)+\paren{\frac{d\log D}{\log\log D}}\log C}}
=\left(\frac{1}{d\log D}-o(1)\right)\log N,\]
we conclude that 
\[h(N)\gtrsim N^{\paren{\frac1{d\log D}-o(1)}(\log\log D+\log d-\log\log\log D)}= N^{(1-o(1))\frac{\log \log D}{d\log D}}.\]

Finally, it is well known that every cograph is perfect and thus every $h$-vertex cograph contains a clique or independent set of size $\sqrt{h}$. Therefore we conclude that every $N$-vertex semialgebraic graph contains a clique or independent set of size at least $\sqrt{h(N)}$, as desired.
\end{proof}

We now construct semialgebraic graphs with no large cliques or independent sets. Taking the $r$-fold iterated blowup of a random graph of an appropriate size produces an $N$-vertex semialgebraic graph in $\R^2$ of total degree $D$ with no clique or independent set of size $N^{(2+o(1))\log\log D/\log D}$. To give an efficient semialgebraic representation of this iterated blowup, we represent an $r$-tuple by viewing it as a single number written in base $B$. A similar base representation trick was earlier used in a construction of Conlon, Fox, Pach, Sudakov, and Suk \cite[Section 3.4]{CFPSS14}. 

To get the correct $d$-dependence in the exponent, we want to more efficiently represent this iterated blowup as a semialgebraic graph in $\R^d$ as $d$ grows. We only know how to do this if the original graph that we blow up has an extra property. (See property $(i)$ in \cref{lem:ramsey-grid} below.)

\begin{proposition}
\label{thm:erdos-hajnal-upper}
For $d\geq 2$ even, for $D$ sufficiently large in terms of $d$, and for $N\geq D^d$, there exists an $N$-vertex semialgebraic graph in $\R^d$ with total degree $D$ that does not contain a clique or independent set of size at least
\[N^{\paren{4+o(1)}\frac{\log\log D}{d\log D}}\]
where the $o(1)$ term goes to zero for each fixed $d$ as $D$ goes to infinity.
\end{proposition}

\begin{proof}
First define $d=2k$. The graph we produce will be an iterated blowup of the graph supplied by the following claim. For $a,b\in S^k$, define $\Phi(a,b)\in(S^2\cup\{\ast\})^k$ by $\Phi_i(a,b)=(a_i,b_i)$ if $a_i\neq b_i$ and $\Phi_i(a_i,b_i)=\ast$ otherwise.

\begin{claim}
\label{lem:ramsey-grid}
For a set $S$ of size $t\geq2$ and $k\geq 1$, there exists a graph $G=(S^k, E)$ with the following properties:
\begin{enumerate}[(i)]
    \item for vertices $a,b,c,d\in S^k$ with $\Phi(a,b)=\Phi(c,d)$ we have $\{a,b\}\in E$ if and only if $\{c,d\}\in E$; and
    \item any clique or independent set in $G$ has size less than $O_k(\log t)$.
\end{enumerate}
\end{claim}
\begin{proof}
Choose a graph $G$ that satisfies $(i)$ uniformly at random. We will show by induction on $k$ that for any subset $U$ of $S^k$, the probability that $G[U]$ is complete is upper bounded by $2^{1-\Omega_k(\abs{U}^2)}$. Consider a family of pairs $\{a,b\}$, such that $a_i\neq b_i$ for all $i\in[k]$ for each of these pairs. Note that each pair is an edge with probability $1/2$ and these events are mutually independent. Therefore, if there are at least $\abs{U}^2/4-1$ pairs $\{a,b\}\in \binom{U}{2}$ with $a_i\neq b_i$ for all $i\in[k]$, then we are done. In particular, we are done if $k=1$.
    
For the remaining case, there exists an $i\in[k]$ such that there are at least $\Omega_k(\abs{U}^2)$ pairs $\{a,b\}\in\binom{U}{2}$ with $a_i=b_i$. Picking some $v\in U$ that participates in $\Omega_k(|U|)$ of these pairs, we get that the set $U'=\{a\in U: a_i=v_i\}$ has size at least $\Omega_k(\abs{U})$. Note crucially that $G[U']$ has the same distribution as the graph constructed the same way on $S^{k-1}$. Therefore by the inductive hypothesis
\[\PP\left(G[U]\textup{ is a clique}\right)\leq \PP\left(G[U']\textup{ is a clique}\right) \leq 2^{1-\Omega_k(\abs{U'}^2)} \leq 2^{1-\Omega_k(\abs{U}^2)}.\]
This completes the induction.

Now let $C_k$ be some constant to be determined. The probability that there exists some clique in $G$ of size at least $C_k\log t$ can be upper bounded by 
\[\binom{t^k}{C_k\log t}2^{1-\Omega_k(C_k^2\log^2t)} \leq \exp\left(kC_k\log^2t+1-\Omega_k(C_k^2\log^2 t)\right).\]
The probability that there exists some clique or independent set of size at least $C_k\log t$ is at most twice this quantity. Therefore, as long as $C_k$ is chosen sufficiently sufficiently large with respect to the implicit constant, this probability is less than $1$. As a consequence, there is some $G$ satisfying $(i)$ and $(ii)$.
\end{proof}

For an appropriate value of $t\geq 2$ let $S\subset \{1,\ldots,2t^2\}$ be a Sidon set of size $t$. Let $G_0=(S^k,E_0)$ be a graph that satisfies the two conditions in \cref{lem:ramsey-grid}. For an appropriate value of $r\geq 1$, let $G$ be the $r$-fold iterated blowup of $G_0$. This means that $G$ has vertex set $(S^k)^r$ where $(a_1,\ldots,a_r)$ is adjacent to $(b_1,\ldots, b_r)$ if there exists $i$ such that $a_j=b_j$ for $j<i$ and $a_i$ is adjacent to $b_i$ in $G_0$. For any clique or independent set in $G$, its projection onto the first coordinate is a clique or independent set in $G_0$. By induction on $r$, this implies that $G$ does not contain a clique or independent set of size $(C_k\log t)^r$ for some constant $C_k$ depending only on $k$.

Now we claim that $G$ is a semialgebraic graph in $\R^d$ with total degree $4kt^2+2k^2$. Define $B=20t^2$. For each vertex $a\in (S^k)^r$, represent $a$ by the point $(x_{a,1},\ldots,x_{a,k},y_{a,1},\ldots,y_{a,k})\in\R^{2k}$ where
\[x_{a,i}=\sum_{j=1}^r a_{j,i} B^{r-j}\qquad\text{and}\qquad y_{a,i}=\sum_{j=1}^r a_{j,i} B^{2(r-j)}.\]
In other words, we view $a$ as a $k$-tuple of $r$-digit numbers and we view each of these numbers as either written in base $B$ or base $B^2$.

For each pair $p>q$ of elements of $S$ and each $i\in[k]$, consider the polynomials
\begin{align*}
(x_i-x_i')^2-\paren{p-q+\frac12}(y_i-y_i')&\\
(x_i-x_i')^2-\paren{p-q-\frac12}(y_i-y_i')&\\
(x_i-x_i')^2+\paren{p-q-\frac12}(y_i-y_i')&\\
(x_i-x_i')^2+\paren{p-q+\frac12}(y_i-y_i')&
\end{align*}
as well as, for each $i\in[k]$, the polynomial
\[y_i-y_i',\]
and finally, for each $1\leq i_1<i_2\leq k$, the polynomials
\begin{align*}
(y_{i_1}-y'_{i_1})-B(y_{i_2}-y'_{i_2})&\\
(y_{i_1}-y'_{i_1})+B(y_{i_2}-y'_{i_2})&\\
B(y_{i_1}-y'_{i_1})-(y_{i_2}-y'_{i_2})&\\
B(y_{i_1}-y'_{i_1})+(y_{i_2}-y'_{i_2})&
\end{align*}

These polynomials have total degree $k\paren{8\binom t2+1}+4\binom k2<4kt^2+2k^2$. We now claim that one can determine if $\{a,b\}$ is an edge of $G$ from the signs of these polynomials evaluated at $(x_1,\ldots,x_k,y_1,\ldots,y_k,x'_1,\ldots,x'_k,y'_1,\ldots,y'_k)=(x_{a,1},\ldots,x_{a,k},y_{a,1},\ldots,y_{a,k},x_{b,1},\ldots,x_{b,k},y_{b,1}\ldots,y_{b,k})$.

First, for each $i\in [k]$, if there exists a pair $(p_i,q_i)$ of elements of $S$ such that
\[|p_i-q_i|-\frac12<\abs{\frac{(x_{a,i}-x_{b,i})^2}{y_{a,i}-y_{b,i}}}<|p_i-q_i|+\frac12,\]
then one can determine $|p_i-q_i|$. To see why this is useful, define $j_i$ to be the minimum positive integer such that $a_{j_i,i}\neq b_{j_i,i}$. Suppose that $a_{j_i,i}>b_{j_i,i}$. Then
\[(a_{j_i,i}-b_{j_i,i})B^{r-j_i}-2t^2 B^{r-j_i-1}\leq x_{a,i}-x_{b,i}\leq (a_{j_i,i}-b_{j_i,i})B^{r-j_i}+2t^2 B^{r-j_i-1}\]
and
\[(a_{j_i,i}-b_{j_i,i})B^{2(r-j_i)}-2t^2 B^{2(r-j_i)-1}\leq y_{a,i}-y_{b,i}\leq (a_{j_i,i}-b_{j_i,i})B^{2(r-j_i)}+2t^2 B^{2(r-j_i)-1}.\]
In particular,
\[\frac{(x_{a,i}-x_{b,i})^2}{y_{a,i}-y_{b,i}}\leq\frac{\paren{(a_{j_i,i}-b_{j_i,i})B^{r-j_i}+2t^2 B^{r-j_i-1}}^2}{(a_{j_i,i}-b_{j_i,i})B^{2(r-j_i)}-2t^2 B^{2(r-j_i)-1}}=\frac{\paren{a_{j_i,i}-b_{j_i,i}+\frac1{10}}^2}{a_{j_i,i}-b_{j_i,i}-\frac1{10}}<a_{j_i,i}-b_{j_i,i}+\frac12.\]
A similar lower bound holds.
Therefore if such a pair $(p_i,q_i)$ exists, then $p_i-q_i = a_{j_i,i}-b_{j_i,i}$. Since $S$ is a Sidon set, one can then determine the pair $(a_{j_i,i},b_{j_i,i})$ from $a_{j_i,i}-b_{j_i,i}$. 

Next we claim that for each $1\leq i_1<i_2\leq k$, we can determine which of $j_{i_1}<j_{i_2}$ or $j_{i_1}=j_{i_2}$ or $j_{i_1}>j_{i_2}$ holds. To see this, notice that
\[\abs{\frac{y_{a,i_1}-y_{b,i_1}}{y_{a,i_2}-y_{b,i_2}}}\geq\frac{\abs{a_{j_{i_1},i_1}-b_{j_{i_1},i_1}}-2t^2B^{-1}}{\abs{a_{j_{i_2},i_2}-b_{j_{i_2},i_2}}+2t^2B^{-1}}\cdot B^{2(j_{i_2}-j_{i_1})}\geq\frac{1-\frac{1}{10}}{2t^2+\frac{1}{10}} B^{2(j_{i_2}-j_{i_1})}>B^{2(j_{i_2}-j_{i_1})-1}\]
and similarly
\[\abs{\frac{y_{a,i_1}-y_{b,i_1}}{y_{a,i_2}-y_{b,i_2}}}\leq \frac{2t^2+\frac{1}{10}}{1-\frac{1}{10}}B^{2(j_{i_2}-j_{i_1})}<B^{2(j_{i_2}-j_{i_1})+1}.\]
Hence we see that
\begin{align*}
\abs{\frac{y_{a,i_1}-y_{b,i_1}}{y_{a,i_2}-y_{b,i_2}}}>B\qquad&\text{if }j_{i_1}<j_{i_2},\\
B^{-1}<\abs{\frac{y_{a,i_1}-y_{b,i_1}}{y_{a,i_2}-y_{b,i_2}}}<B\qquad&\text{if }j_{i_1}=j_{i_2},\\
\abs{\frac{y_{a,i_1}-y_{b,i_1}}{y_{a,i_2}-y_{b,i_2}}}<B^{-1}\qquad&\text{if }j_{i_1}>j_{i_2}.
\end{align*}

Therefore from the signs of the above polynomials, one can determine the relative order of $j_1,\ldots,j_k$. Combining this with the above determination of $a_{j_1,1},\ldots,a_{j_k,k},b_{j_1,1},\ldots,b_{j_k,k}$, we have determined the value of $\Phi(a_j,b_j)$, where $j=\min\{j_1,\ldots,j_k\}$. By the choice of $G_0$, whether $\{a_j,b_j\}$ is an edge of $G_0$ is determined by $\Phi(a_j,b_j)$ and by the definition of an iterated blowup, whether $\{a,b\}$ is an edge of $G$ is determined from this. 

Now define $t=\floor{\tfrac12\sqrt{\tfrac {D-2k^2}k}}$ and $r=\floor{\tfrac{\log N}{k\log t}}$. Note that $r\geq 1$ and $t$ is sufficiently large by our assumptions on $N,D$. Then the graph $G$ constructed above is a semialgebraic graph in $\R^d$ on at most $N$ vertices with total degree at most $D$. Furthermore, $G$ does not contain a clique or independent set of size
\[\begin{split}
(C_k\log t)^r
&= \exp\paren{r(\log\log t+\log C_k)}
\leq \exp\paren{\frac{\log N}{k\log t}(1+o(1))\log\log D}\\
&\leq\exp\paren{(1+o(1))\frac{\log N\log\log D}{\tfrac12k\log D}}
=N^{(4+o(1))\frac{\log\log D}{d\log D}}.\qedhere
\end{split}\]
\end{proof}

\subsection{Property testing}
Property testing is an area of theoretical computer science that aims to find extremely fast randomized algorithms that determine if an object has some property or is far from satisfying the property. Typically the goal is to find an algorithm that only makes a constant number of queries to the object. In the dense graph model one can query if a pair of vertices form an edge or not. We say that a property $\cP$ is testable if for all $\epsilon>0$ there exists an algorithm that makes $O_\epsilon(1)$ queries to a graph $G$ and, if $G\in\cP$, then the algorithm accepts, and if $G$ is $\epsilon$-far from $\cP$, then the algorithm rejects with probability at least 2/3. We say that an $N$-vertex graph is $\epsilon$-close to $\cP$ if by adding or removing at most $\epsilon N^2$ edges, $G$ can be turned into an element of $\cP$; otherwise $G$ is $\epsilon$-far from $\cP$.

Alon and Shapira classified the testable graph properties as exactly those that satisfy a combinatorial property known as being semi-hereditary \cite{AS08}. This includes monotone properties, such as $H$-freeness, as well as hereditary properties, such as induced $H$-freeness, as well as many additional properties which may be more pathological. The strength of this result lies in the fact that the number of queries is independent of the size of the graph. However, it can an extremely large function of $\epsilon$. For properties such as $H$-freeness and induced $H$-freeness the number of queries has tower-type growth as a function of $\epsilon$. For arbitrary semi-hereditary properties, the number of queries can grow essentially arbitrarily quickly for some pathological properties.

When restricted to the class of semialgebraic graphs and hypergraphs of bounded dimension and total degree, the situation is quantitatively nicer. Properties such as $H$-freeness and induced $H$-freeness can be tested with $\epsilon^{-O(1)}$ queries while general hereditary properties can be testing with a number of queries that is polynomial in a function that captures the structure of the property \cite{FPS16}. These results also generalize exactly to hypergraphs. All of these results follow from Fox, Pach, and Suk's regularity lemma for semialgebraic hypergraphs. With our quantitatively improved regularity lemma, we can improve the number of queries required. For simplicity, we only carry out this analysis for $H$-freeness, though we expect our results to be able to say something about the more general problem.

\begin{proposition}
Let $H$ be a $k$-uniform hypergraph. In the class of $k$-uniform semialgebraic hypergraphs in $\R^d$ with total degree at most $D$, the property of $H$-freeness can be tested with $O_{d,k,|H|}(\epsilon^{-(d+1)|H|}D^{d|H|})$ queries. 
\end{proposition}

There is a close relationship between property testing results and removal lemmas. Via standard arguments, the above proposition follows from (and is essentially equivalent to) the upcoming removal lemma. The property testing algorithm is very simple. Sample $O_{d,k,|H|}(\epsilon^{-(d+1)|H|}D^{d|H|})$ random injections from $H$ into the target hypergraph $G$ and check if any of them correspond to a copy of $H$ in $G$. If any do, the algorithm rejects; otherwise it accepts. The following removal lemma shows the correctness of this algorithm.

\begin{proposition}
For $\epsilon>0$, there exists $\delta=\Omega_{d,k,n}(\epsilon^{(d+1)n}D^{-dn})$ such that the following holds. Let $H$ be an $n$-vertex $k$-uniform hypergraph and let $G$ be an $N$-vertex $k$-uniform semialgebraic hypergraph in $\R^d$ of total degree at most $D$. If $G$ has fewer than $\delta N^n$ labelled copies of $H$, then $G$ can be made $H$-free by removing at most $\epsilon N^k$ edges.
\end{proposition}

\begin{proof}
Let $G=(V,E)$. We apply multilevel polynomial partitioning to $V$ with parameter $A=C_{d,k}D/\epsilon$. This produces a partition $\Pi$ of $V$ with $O_{d,k}((D/\epsilon)^d)$ parts. By the same argument as in the proof of \cref{thm:main}, all but an $(\epsilon/2)$-fraction of $k$-tuples of vertices lie in a $k$-tuple of parts which is homogeneous. (To be clear, we say that a $k$-tuple of not necessarily distinct parts $(\pi_1,\ldots,\pi_k)$ is complete if every $k$-tuples of vertices $(v_1,\ldots,v_k)\in\pi_1\times\cdots\times\pi_k$ for which $v_1,\ldots, v_k$ are distinct is an edge and the analogous definition for an empty $k$-tuple.)

Now we define a ``cleaned up'' hypergraph $\tilde G$ on the same vertex set of $G$. We produce $\tilde G$ by first removing any edge that lies an a $k$-tuple of parts that is not homogeneous. We also remove any edge that lies in a $k$-tuple of parts $(\pi_1,\ldots,\pi_k)$ where any of the parts have size less than $\epsilon N/(2|\Pi|k)$. In total we have removed at most $\epsilon N^k$ edges of $G$. 

We prove the contrapositive. Suppose that $G$ cannot be made $H$-free by removing at most $\epsilon N^k$ edges. Then $\tilde G$ must contain a copy of $H$, say between parts $\pi_1,\ldots,\pi_n$. By the construction of $\tilde G$, we know that if vertex $i$ of $H$ is non-isolated then $|\pi_i|\geq\epsilon N/(2|\Pi|k)$ and for each edge $\{i_1,\ldots,i_k\}$ of $H$, the tuple $(\pi_{i_1},\ldots,\pi_{i_k})$ is complete in $G$. Now we can lower bound the number of labelled copies of $H$ in $G$ by the number of copies where each non-isolated vertex of $H$ maps into the corresponding part of the partition (and the isolated vertices map into arbitrary parts). This number is at least
\[\paren{\frac{\epsilon N}{2|\Pi|k}}\paren{\frac{\epsilon N}{2|\Pi|k}-1}\cdots\paren{\frac{\epsilon N}{2|\Pi|k}-(n-1)}\geq \delta N^n.\]
The last inequality holds for an appropriate $\delta=\Omega_{d,k,n}(\epsilon^{(d+1)n}D^{-dn})$ if $N\geq N_0$ where we define $N_0=4kn|\Pi|\epsilon^{-1}$. For $N<N_0$, the desired statement holds for $\delta=N_0^{-n}$ for all $\epsilon$, completing the proof.
\end{proof}

\section{Lower bound constructions}
\label{sec:lower-bound}

We give three examples of semialgebraic graphs of increasing complexity. The first shows that the regularity lemma, \cref{thm:main}, has optimal $\epsilon$-dependence, while the third shows that the dependence on both $\epsilon$ and $D$ is optimal for all $d$.

\begin{example}
\label{thm:eps-example}
For every $m,d\geq 1$, let $P_1=P_2=[m]^d$. Define the graph $G=(P_1\sqcup P_2,E)$ where $x$ is adjacent to $y$ if $x_i=y_i$ for some $i\in[d]$. This graph is semialgebraic in $\R^d$ of total degree $d$. 
\end{example}

For any $\epsilon>0$ and $m$ sufficiently large in terms of $d,\epsilon$, we will show that any homogeneous partition of the above graph with error $\epsilon$ must have $\Omega_d(\epsilon^{-d})$ parts. This shows that the $\epsilon$-dependence in \cref{thm:main} is optimal.

For our next examples, for each $k\geq1$ we define real numbers $a_0,\ldots,a_{k-1}\in[0,1)$ as follows. Let $p$ be the smallest prime greater than $k$. Then define $a_i\in[0,1)$ so that $i^2/p\equiv a_i\pmod 1$. 

\begin{example}
\label{thm:D-eps-example-1d}
For $D,m\geq 1$, let $a_0,\ldots,a_{D-1}\in[0,1)$ be the real numbers defined above. Set $P_1=[m]$ and $P_2=[2m]$. Define the graph $G=(P_1\sqcup P_2,E)$ where $x$ is adjacent to $y$ if an even number of the inequalities
\[y-x-\frac{im}{D}-\frac{a_im}{3D}\geq 0\qquad\text{for }0\leq i<D\]
hold. This graph is semialgebraic in $\R^1$ of total degree $D$.
\end{example}

The edges in $G$ correspond to pairs $(x,y)$ lying in the black region in the following diagram.

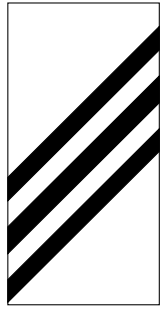
\begin{center}
\begin{tikzpicture}[scale = 2]
\def\a{0.3};
\def\b{0};
\def\c{0.1};
\def\d{0.3};
\def\e{0.4};
\def\f{0.2};
\draw (0,0) -- (1,0) -- (1,2) -- (0,2) -- cycle;
\filldraw (0,0/6+\a/18) -- (1,1+0/6+\a/18) -- (1,1+1/6+\b/18) -- (0,0+1/6+\b/18) -- cycle;
\filldraw (0,2/6+\c/18) -- (1,1+2/6+\c/18) -- (1,1+3/6+\d/18) -- (0,0+3/6+\d/18) -- cycle;
\filldraw (0,4/6+\e/18) -- (1,1+4/6+\e/18) -- (1,1+5/6+\f/18) -- (0,0+5/6+\f/18) -- cycle;
\end{tikzpicture}
\end{center}

For any $\epsilon>0$, any $D\geq 1$, and any $m$ sufficiently large in terms of $D,\epsilon$, we will show that any homogeneous partition of the above graph with error $\epsilon$ must have $\Omega(D/\epsilon)$ parts, showing that the $\epsilon$- and the $D$-dependence in \cref{thm:main} are both optimal when $d=1$. This also generalizes to higher dimensions.

\begin{example}
\label{thm:D-eps-example}
For $D,m,d\geq 1$, set $k=\floor{D/d}$. Let $a_0,\ldots,a_{k-1}\in[0,1)$ be the real numbers defined above. Set $P_1=[m]^d$ and $P_2=[2m]^d$. Define the graph $G=(P_1\sqcup P_2,E)$ where $x$ is adjacent to $y$ if an even number of the inequalities
\[y_j-x_j-\frac{im}{k}-\frac{a_im}{3k}\geq 0\qquad\text{for }0\leq i<k\text{ and }j\in[d]\]
hold. This graph is semialgebraic in $\R^d$ of total degree $kd\leq D$. This graph can either be written in terms of $kd$ polynomials of degree 1 or in terms of a single polynomial of degree $kd$.
\end{example}

For any $\epsilon>0$, any $D\geq 2d$, and any $m$ sufficiently large in terms of $D,\epsilon,d$, we will show that any homogeneous partition of the above graph with error $\epsilon$ must have $\Omega_d((D/\epsilon)^d)$ parts, showing that both the $\epsilon$- and the $D$-dependencies in \cref{thm:main} are optimal.

\begin{proposition}
For every $d\geq 1$ and every $\epsilon>0$, for $m$ sufficiently large, the graph in \cref{thm:eps-example} has the property that any homogeneous partition with error $\epsilon$ must have at least $(2\epsilon)^{-d}$ parts.
\end{proposition}

\begin{proof}
Suppose $m\geq d\epsilon^{-(d+1)}$. Let $G$ be the graph in \cref{thm:eps-example}. For a set $A\subseteq P_1$, write $N(A)$ for the neighborhood of $A$ in $G$, i.e., the set of vertices in $P_2$ that are adjacent to at least one vertex in $A$. Also write $\pi_j(A)\subseteq [m]$ for the projection of $A$ onto the $j$th coordinate axis, i.e., the set of $a\in [m]$ such that there exists $x\in A$ with $x_j=a$.

\begin{claim}
\label{thm:g-expands}
For $A\subseteq P_1$,
\[|N(A)|\geq m^d-(m-|A|^{1/d})^d\geq |A|^{1/d}m^{d-1}.\]
\end{claim}
\begin{proof}
For $y\in P_2$, we have $y\not\in N(A)$ if and only if $y_j\not\in \pi_j(A)$ for all $j$. Thus we see that
\[|N(A)|
=m^d-\prod_{j=1}^d (m-|\pi_j(A)|)
\geq m^d-\paren{m-\frac{\sum_{j=1}^d|\pi_j(A)|}{d}}^d
\geq m^d-\paren{m-\paren{\prod_{j=1}^d|\pi_j(A)|}^{1/d}}^d.
\]
by two applications of the AM-GM inequality.

Since $A\subseteq \pi_1(A)\times\cdots\times\pi_d(A)$, we see that $\prod_{j=1}^d|\pi_j(A)|\geq|A|$, so $|N(A)|\geq m^d-(m-|A|^{1/d})^d$. To deduce the last inequality in the lemma statement, note that
\[1-\frac{|A|^{1/d}}m\geq \paren{1-\frac{|A|^{1/d}}m}^d\]
since the left-hand side is at most 1 and $d\geq 1$. Multiplying by $m^d$ and rearranging gives the desired inequality.
\end{proof}

Suppose for contradiction that $P_1=A_1\sqcup\cdots\sqcup A_K$ and $P_2=B_1\sqcup\cdots\sqcup B_K$ are partitions with $K< (2\epsilon)^{-d}$ parts that form a homogeneous partition of $G$ with error $\epsilon$. Now the graph $G$ is regular of degree $m^d-(m-1)^d\leq dm^{d-1}$. Therefore, if $E[A_i\times B_j]$ is complete, then $|A_i|,|B_j|\leq dm^{d-1}$. Thus
\[\sum_{i,j:E[A_i\times B_j]\text{ non-empty}}\frac{|A_i||B_j|}{m^{2d}}\leq\epsilon+\frac{K^2d^2}{m^2}\leq 2\epsilon,\]
where the last inequality follows from the assumption that $m$ is sufficiently large.
Furthermore, we can bound
\begin{align*}
\sum_{i,j:E[A_i\times B_j]\text{ non-empty}}\frac{|A_i||B_j|}{m^{2d}}
&\geq\sum_{i=1}^K \frac{|A_i||N(A_i)|}{m^{2d}}\\
&\geq\sum_{i=1}^K\frac{|A_i|^{1+\tfrac1d}m^{d-1}}{m^{2d}}\tag*{\tiny[\cref{thm:g-expands}]}\\
&=\frac{K}{m^{d+1}}\paren{\frac1K\sum_{i=1}^K|A_i|^{1+\tfrac1d}}\\
&\geq \frac{K}{m^{d+1}}\paren{\frac1K\sum_{i=1}^K|A_i|}^{1+\tfrac1d}\tag*{\tiny[Jensen's inequality]}\\
&=\frac1{K^{1/d}}
\end{align*}
Combining this lower bound with the previous upper bound, we conclude that $K\geq (2\epsilon)^{-d}$, as desired.
\end{proof}

We see that understanding the behavior of neighborhoods plays an important role in the proof.
In the proof that \cref{thm:D-eps-example} does not have small homogeneous partitions, we will show that neighborhoods of far-away points differ a lot, which will again be key in the proof.

\begin{proposition}
For every $D\geq 2d$ and every $\epsilon\in(0,1/1000)$, for $m$ sufficiently large, the graph in \cref{thm:D-eps-example} has the property that any homogeneous partition with error $\epsilon$ must have at least $(D/(32\epsilon d))^d$ parts.
\end{proposition}

\begin{proof}
Suppose $m\geq 4D\epsilon^{-1}$.  Let $G$ be the graph in \cref{thm:D-eps-example} with parameters $D,m,d$. Recall that $k=\floor{D/d}\geq 2$. Let $H$ be the graph in \cref{thm:D-eps-example-1d} with parameters $k,m$. We first show that the numbers $a_0,\ldots,a_{k-1}$ in the definitions of $G$ and $H$ are well-spaced in the following way.

\begin{claim}
\label{thm:a-well-spaced}
For each $1\leq \ell\leq k$, the number of pairs $(i,i')\in\{0,1\ldots,k-\ell-1\}^2$ with $i\neq i'$ and satisfying
\[|a_{i+\ell}-a_i-a_{i'+\ell}+a_{i'}|<\frac{1}{16}\]
is at most $k^2/4-k/2$.
\end{claim}
\begin{proof}
This is easy to check for $k=2,3$. Otherwise, we have $k<p<2k-2$ since $p$ is the smallest prime greater than $k$. Now if $(i,i')$ satisfies the above condition, then
\[a_{i+\ell}-a_i-a_{i'+\ell}+a_{i'}\equiv \frac{(i+\ell)^2-i^2-(i'+\ell)^2+i'^2}{p}\equiv\frac{2\ell(i-i')}{p}\pmod 1\]
is within $1/16$ of an integer. Since $2\ell\not\equiv 0\pmod p$, this implies that there are at most $2\floor{p/16}$ possible values of $i-i'\neq 0$ and thus at most $k(2\floor{p/16})\leq k(p-1)/8\leq k(2k-4)/8=k^2/4-k/2$ possible values for $(i,i')$.
\end{proof}

We next show that vertices that are far apart have very different neighborhoods in $H$ and also that these neighborhoods are not close to complementary.

\begin{claim}
\label{thm:1-dim-dist}
For $x,x'\in[m]$,
\[|N_H(x)\Delta N_H(x')|,2m-|N_H(x)\Delta N_H(x')|\geq \min\set{\frac{k(|x-x'|-1)}{2},\frac{m}{200}}.\]
\end{claim}
\begin{proof}
Define the two sets
\[S=\set{x+\frac{im}k+\frac{a_{i}m}{3k}:0\leq i<k},\]
and
\[T=\set{x'+\frac{im}k+\frac{a_{i}m}{3k}:0\leq i<k}.\]

Since $a_i\in[0,1)$, the elements of $S$ are each separated by at least $2m/3k$. The same holds for the elements of $T$. We claim that there are at most $k/2$ pairs $(i,i')\in\{0,\ldots,k-1\}^2$ such that the $i$-th element of $S$ and the $i'$-th element of $T$ are at distance smaller than $\min\{|x-x'|,m/(96k)\}$. For such a pair $(i,i')$, we have
\begin{equation}
\label{eq:i-i'-close}
\abs{\paren{x+\frac{im}k+\frac{a_im}{3k}}-\paren{x'+\frac{i'm}{k}+\frac{a_{i'}m}{3k}}}<\min\set{|x-x'|,\frac{m}{96k}}.
\end{equation}
Note that if $i=i'$, the above inequality would imply that $|x-x'|<|x-x'|$ which does not hold. Thus we have $i\neq i'$.

Now \cref{eq:i-i'-close} rearranges to the containment
\[\frac km(x-x')\in \paren{i'-i+\frac{a_{i'}-a_i}{3}-\frac{1}{96},i'-i+\frac{a_{i'}-a_i}{3}+\frac{1}{96}}\subset\paren{i'-i-\frac1{2},i'-i+\frac12}.\]
Thus we see that $i'-i=\ell\neq 0$ where $\ell$ is the closest integer to $(k/m)(x-x')$. Thus the original inequality implies that
\[\abs{x-x'-\frac{\ell m}{k}-\frac{(a_{i+\ell}-a_i)m}{3k}}<\frac{m}{96k}.\]
If $(i,i+\ell)$ and $(i',i'+\ell)$ both satisfy \cref{eq:i-i'-close}, we then have
\[|a_{i+\ell}-a_i-a_{i'+\ell}+a_{i'}|\leq \frac {3k}m\abs{x-x'-\frac{\ell m}{k}-\frac{(a_{i+\ell}-a_i)m}{3k}}-\frac {3k}m\abs{x-x'-\frac{\ell m}{k}-\frac{(a_{i'+\ell}-a_{i'})m}{3k}}<\frac{1}{16}.\]
Therefore by \cref{thm:a-well-spaced}, we see that there are at most $k/2$ pairs of elements of $S\sqcup T$ that are at distance smaller than $\min\{|x-x'|,m/(96k)\}$.

The $2k$ elements of $S\sqcup T$ divide the interval $[0,2m]$ into $2k+1$ subintervals. The first quantity, $|N_H(x)\Delta N_H(x')|$, is the number of points of $[2m]$ which lie in the even numbered subintervals. The second quantity, $2m-|N_H(x)\Delta N_H(x')|$ is the number of points which lie in the odd numbered subintervals. Thus we are summing the number of integer points in at least $k$ subintervals. Of all the subintervals, we have shown that at most $k/2$ have length strictly smaller than $\min\{|x-x'|,m/(96k)\}$ so we end up summing at least $k/2$ subintervals, each containing at least $\min\{|x-x'|,m/(96k)\}-1$ integer points. Therefore we have shown that
\[|N_H(x)\Delta N_H(x')|,2m-|N_H(x)\Delta N_H(x')|\geq\frac{k}2\paren{\min\set{|x-x'|,\frac m{96k}}-1}.\]
This proves the desired inequality since $m\geq 4000k$.
\end{proof}

Now we claim that for $x,x'\in [m]^d$ we have
\begin{equation}
\label{eq:d-dim-dist}
|N_G(x)\Delta N_G(x')|\geq (2m)^{d-1}\cdot\min\set{\frac{k(\|x-x'\|_\infty-1)}{2},\frac{m}{200}}.
\end{equation}
To see this, fix $j\in [d]$ such that $|x_j-x'_j|=\|x-x'\|_\infty$. For any $z\in[2m]^{[d]\setminus\{j\}}$, consider the line $L_z=\{y\in[2m]^d:y_i=z_i\text{ for all }i\in[d]\setminus\{j\}\}$. Clearly these $(2m)^{d-1}$ lines partition $P_2$. Furthermore, it follows from the definition of $G$ and $H$ that for each $z$ either $L_z\cap N_G(x)$ is $N_H(x_j)$ or its complement. Now by \cref{thm:1-dim-dist} we see that $L_z\cap (N_G(x)\Delta N_G(x'))$ is large for each $z$. Summing over $z$ gives \cref{eq:d-dim-dist}.

To complete the proof, suppose for contradiction that $P_1=A_1\sqcup\cdots\sqcup A_K$ and $P_2=B_1\sqcup\cdots\sqcup B_K$ are partitions with $K\leq (D/(32d\epsilon))^{-d}$ parts that form a homogeneous partition of $G$ with error $\epsilon$. Pick points $x_i,x'_i\in A_i$ that achieve $\max_{x_i,x'_i\in A_i}\|x_i-x'_i\|_\infty$. Note that $|A_i|\leq (\|x'_i-x_i\|_\infty+1)^d$ and also for any part $B_j$ with $y\in B_j$ such that $y\in N_G(x_i)\Delta N_G(x'_i)$, we know $E[A_i\times B_j]$ is not homogeneous. Thus
\begin{align*}
\epsilon 
& \geq \sum_{i,j:E[A_i\times B_j]\text{ not homogeneous}}\frac{|A_i||B_j|}{m^{d}(2m)^d}\\
& \geq \sum_i \frac{|A_i||N_G(x_i)\Delta N_G(x'_i)|}{m^{d}(2m)^d}\\
& \geq \sum_i \min\set{\frac{(2m)^{d-1}|A_i|k(\|x'_i-x_i\|_\infty-1)}{2m^{d}(2m)^d},\frac{(2m)^{d-1}|A_i|m}{200m^d(2m)^d}}\tag*{\tiny[by \cref{eq:d-dim-dist}]}\\
& \geq \sum_i \min\set{\frac{k|A_i|(|A_i|^{1/d}-2)}{4m^{d+1}},\frac{|A_i|}{400m^d}}\\
& = \sum_i\min\set{\frac k4 \frac{|A_i|^{1+\tfrac1d}}{m^{d+1}},\paren{\frac{1}{400}+\frac k{2m}} \frac{|A_i|}{m^d}}-\frac k{2m}.
\end{align*}
Write $[K]=I\sqcup J$ where $i\in I$ if the first term achieves the minimum and $i\in J$ otherwise. Recalling that $m$ is sufficiently large, we conclude that $\sum_{i\in J}|A_i|/m^d\leq 500\epsilon$, so $\sum_{i\in I}|A_i|/m^d\geq 1-500\epsilon\geq 1/2$. By Jensen's inequality, we have
\[\epsilon\geq \frac k4\sum_{i\in I}\frac{|A_i|^{1+\tfrac1d}}{m^{d+1}}\geq \frac {k}4|I|\paren{\frac1{|I|}\sum_{i\in I}\frac{|A_i|}{m^d}}^{1+\frac1d}\geq\frac{k}{16|I|^{1/d}}\geq\frac{k}{16K^{1/d}}.\]
Since $k\geq D/(2d)$, this rearranges to the desired result.
\end{proof}

\section{Open problems}
\label{sec:open-problems}

Our regularity lemma produces partitions with $O_{d,k}((D/\epsilon)^d)$ parts. Tracing through the dependencies, the hidden constant is $\exp\paren{(dk)^{O(1)}}$. In contrast we prove a lower bound of $(32d)^{-d}$ on this constant. It seems likely that the dependence on dimension can be improved in some of the polynomial method tools we use from \cite{Wal20} which would start to close this gap.

\begin{problem}
Improve the dependence of the number of parts on $d$ and $k$ in the regularity lemma, \cref{thm:main}.
\end{problem}

To deduce the equitable version of the regularity lemma, we lose a factor of $\epsilon$. In dimension $d=1$, we can produce the equitable partitions by hand with no loss. It would be interesting if one could avoid this loss in all dimensions.

\begin{conjecture}
\cref{thm:main-equitable} holds with partitions into $O_{d,k}((D/\epsilon)^d)$ parts.
\end{conjecture}

We do not know if either our general Tur\'an and Zarankiewicz results are tight except for a very small number of examples.

\begin{question}
What is the correct exponent of $\epsilon$ in the semialgebraic Tur\'an result, \cref{thm:turan}?
\end{question}

It is easy to see that \cref{thm:turan} is tight in dimension $d=1$. However, we do not know if it is tight in any other dimension.

For $d=2$, a tight example for the Szemer\'edi--Trotter theorem gives a semialgebraic graph $G_0$ with $n$ vertices at $Cn^{4/3}$ edges. This graph is $K_{2,2}$-free. Let $G$ be the $r$-fold blowup of $G_0$. This is an $nr$-vertex graph with $Cn^{4/3}r^2$ edges. The graph $G$ is still semialgebraic in $\R^2$ of total degree $O(1)$. Furthermore, for any complete $S\times T$ in $G$, we see that $\min\{|S|,|T|\}\leq r$. Since $G$ has edge density $\epsilon=Cn^{-2/3}$, this shows that the exponent of 2 in the Tur\'an bound $|S|=\Omega_d (\epsilon^2N)$ and $|T|=\Omega(\epsilon N)$ cannot be replaced by anything smaller than $3/2$.

For the Zarankiewicz problem, the Szemer\'edi--Trotter example shows that the power of $N$ is correct in dimension 2. However for dimensions 3 and higher, the best-known lower bounds are the lower bounds for the point-hyperplane incidence problem \cite{ST23} which do not match the upper bound.

\begin{question}
What is the correct exponent of $N$ in the semialgebraic Zarankiewicz result, \cref{thm:zarankiewicz}?
\end{question}

One can also ask for the correct exponents of $u$ and $D$. When $k=2$, the bound in \cref{thm:zarankiewicz} is $O_d(u^{2/(d+1)}D^{2d/(d+1)}N^{2d/(d+1)})$. We can say at least that neither of the exponents on $u$ or $N$ can be decreased without increasing the other, since for $u=N$ and $D=1$, \cref{thm:zarankiewicz} gives the bound $|E|=O_{d}(N^2)$, which is tight for these parameters. (The bound is tight in the sense that it cannot be improved by $N^c$ for any $c>0$ which is independent of $u$.) 
Moreover, as every $N$-vertex graph is semialgebraic in $\R^d$ of total degree $O_d(N^{1/d})$ we can also say that, for $k=2$, neither of the exponents on $D$ or $N$ can be decreased without increasing the other, since for $k=2$ and $u\sim 1$ and $D\sim N^{1/d}$, \cref{thm:zarankiewicz} gives the bound $|E|=O_d(N^2)$, which is tight. We conjecture that for $k>2$, the exponents on $u$ and $D$ can be decreased so that the bound is tight in this way.

\begin{conjecture}
Let $H$ be a $K^{(k)}_{u,\ldots,u}$-free semialgebraic hypergraph in $\R^d$ with total degree $D$ and parts of size $N$. Then the number of edges of $H$ is at most
\[O_{d,k}\paren{u^{\frac{k}{(k-1)d+1}}D^{\frac{kd}{(k-1)d+1}}N^{k-\frac{k}{(k-1)d+1}}}.\]
\end{conjecture}

For semialgebraic hypergraphs with a particularly simple structure, one can prove the bound $O_{D,d,k}(u^{k/((k-1)d+1)}N^{k-k/((k-1)d+1)})$ by modifying the proof of \cref{thm:zarankiewicz-full}. For example, suppose that there exists a single polynomial $g$ such that $(x_1,\ldots,x_k)\in P_1\times\cdots\times P_k$ is an edge of $H$ if and only if $g(x_1,\ldots,x_k)=0$. In the proof of \cref{thm:zarankiewicz-full}, we considered the $(k-1)$-uniform hypergraph $H'=(P_1\sqcup\cdots\sqcup P_{k-1},T')$ which is the common neighborhood in $H$ of vertices $q_1,\ldots,q_u\in P_k$. In that proof, we used the fact that $H'$ is a semialgebraic hypergraph of total degree $uD$. If $H$ has the above simple form, then instead $H'$ has total degree $O_{D,d,k}(1)$, with no $u$-dependence. (This is a generalization of the observation in the remark after \cref{thm:equilateral-triangle}.) This shows that we can prove the conjectured $u$-dependence at the cost of significantly worse $D$-dependence for these particularly simple semialgebraic hypergraphs.

\appendix

\section{Deduction of polynomial method results}

\subsection{Proof of \texorpdfstring{\cref{thm:barone-basu}}{Theorem 2.2}}

For a variety $V\subseteq \CC^d$ and a point $x\in V(\R)$, we write $\dim_x^\R(V)$ for the local real dimension of $V(\R)$ at $x$. 
Clearly $\dim_x^\R(V)\leq\dim_x^\CC(V)$ for every $x\in V(\R)$.
We will need the following definitions that are made in  \cite[Section 7.1]{Wal20}.

\begin{definition}
    Let $Q = (Q_1,Q_2,\ldots, Q_m)$ be a sequence of polynomials in some ambient space $\CC^d$.
    For any $0\leq j\leq m$, define $Z_j(Q)$ to be $Z(Q_1,\ldots, Q_j)$, and for every $x\in \R^d$  let $\dim^{\R}_{Q,(j)}(x)$ be the $j$-tuple recording the local real dimensions $(\dim^{\R}_x Z_1(Q),\ldots, \dim^{\R}_x Z_j(Q))$.
    Lastly, for any $j$-tuple of non-negative integers $\tau$, define $Z_{\tau}(Q)$ to be the set
    \[\overline{\{x\in \R^d: \dim^{\R}_{Q,(j)}(x) = \tau\}}\]
    where the closure is in the Euclidean topology on $\R^d$.
\end{definition}

\begin{proof}[Proof of \cref{thm:barone-basu}]
We first reorder $Q_1,\ldots, Q_{d-r}$ so that their degrees are non-decreasing.
Now for every $x\in Z(Q)(\R)$ with $\dim_x^{\CC} Z(Q)=r$, by Krull's theorem, we must have that $\dim_x^{\CC} Z_j(Q) = d-j$ for every $j=1,\ldots, d-r$.
As a consequence, we have $\dim^{\R}_{Q,(m)}(x)\leq (d-1,d-2,\ldots, r)$.
Therefore we may now apply \cite[Proposition 7.1]{Wal20} with $j=d-r$ and sum over all $(d-r)$-tuples $\tau$ with $\tau_j\leq d-j$ for all $j\in [d-r]$ to get the desired bound.
\end{proof}

\subsection{Proof of \texorpdfstring{\cref{thm:real-poly-part}}{Theorem 3.3}}
We deduce \cref{thm:real-poly-part} from \cite[Theorem 3.2]{Wal20}. The statement of that theorem involves the quantities $\Delta_i(V)$ and $i_V(A)$. We refer the reader to \cite[Section 3]{Wal20} for the definition of $i_V(A)$ since our proof does not use any property of $i_V(A)$ other than that it is an integer between 0 and $d-\dim V$.

Let $V\subseteq \CC^d$ be an irreducible variety. For every index $0\leq i\leq d-\dim V$, write
\[\Delta_i(V) = \max\left(\frac{\deg V}{\delta_{i+1}(V)\cdots \delta_{d-\dim V}(V)},1\right).\]
Note that in particular, we have $\Delta_{d-\dim V}(V)=\deg V$, and so the middle term in \cref{thm:inverse-bezout} can be rewritten as $\Delta_{d-\dim V}$.
We can extend \cref{thm:inverse-bezout} as follows.
\begin{proposition}\label{prop:general-inverse-bezout}
    For any irreducible variety $V\subseteq \CC^d$ and any index $0\leq i\leq d-\dim V$, 
    \[\prod_{j=1}^{i}\delta_j(V)\lesssim_d \Delta_i(V)\leq \prod_{j=1}^{i}\delta_j(V).\]
\end{proposition}
\begin{proof}
    By \cref{thm:inverse-bezout}, 
    \[\Delta_i(V)\geq \frac{\deg V}{\prod_{j=i+1}^{d-\dim V}\delta_j(V)}\gtrsim_d\frac{\prod_{j=1}^{d-\dim V}\delta_j(V)}{\prod_{j=i+1}^{d-\dim V}\delta_j(V)}=\prod_{j=1}^{i}\delta_j(V).\]
    For the other direction, it clearly holds if $\Delta_i(V)=1$.
    Otherwise, by \cref{thm:inverse-bezout},
    \[\Delta_i(V)=\frac{\deg V}{\prod_{j=i+1}^{d-\dim V}\delta_j(V)}\leq \frac{\prod_{j=1}^{d-\dim V}\delta_j(V)}{\prod_{j=i+1}^{d-\dim V}\delta_j(V)}=\prod_{j=1}^{i}\delta_j(V).\qedhere\]
\end{proof}

\begin{proof}[Proof of \cref{thm:real-poly-part}]
    By directly applying \cite[Theorem 3.2]{Wal20}, it remains to show that
    \[\frac{\abs{P}}{A^{d-i_V(A)}\Delta_{i_V(A)}(V)}\lesssim_d \frac{\abs{P}}{A^{i}\prod_{j=1}^{d-i}\delta_j(V)}.\]
    Since $\Delta_{i_V(A)}(V)\gtrsim_d \prod_{j=1}^{i_V(A)}\delta_j(V)$ by \cref{prop:general-inverse-bezout}, it suffices to show that
    \[A^{d-i_V(A)}\prod_{j=1}^{i_V(A)}\delta_j(V)\geq A^{i}\prod_{j=1}^{d-i}\delta_j(V).\]
    However, by the definition of $i$ and the fact that $\delta_1(V)\leq \cdots \leq \delta_{d-\dim V}(V)$, we know that 
    \[A^{i}\prod_{j=1}^{d-i}\delta_j(V)=\min_{\dim V\leq\ell\leq d}\set{A^{\ell}\prod_{j=1}^{d-\ell}\delta_j(V)}.\]
    Therefore the inequality follows.
\end{proof}

\end{document}